\documentclass[11pt,twoside]{amsart}
\usepackage[T1]{fontenc}
\usepackage{amsmath}
\usepackage{amssymb}
\usepackage{amsthm}
\usepackage{yhmath}
\usepackage{tikz}
\usetikzlibrary{patterns}
\usepackage{setspace}
\usepackage{pinlabel}

\usepackage[all]{xy}\SelectTips{cm}{}

\newtheorem{theorem}{Theorem}[section]
\newtheorem{definitio}[theorem]{Definition}

\newtheorem{rem}[theorem]{Remark}
\newenvironment{remark}{\begin{rem} \rm }{\end{rem}}
\newtheorem{ex}[theorem]{Example}
\newenvironment{example}{\begin{ex} \rm }{\end{ex}}
\newtheorem{lemma}[theorem]{Lemma}
\newtheorem{proposition}[theorem]{Proposition}
\newtheorem{corollary}[theorem]{Corollary}

\newtheorem*{theo}{Theorem}

\newcommand{\fl}[1]{\textrm{\raisebox{-0.14cm}{\begin{tikzpicture}\draw (0.5,0) node{$\longrightarrow$}; \draw (0.5,0.2) node{$#1$};
  \end{tikzpicture}}}}

\allowdisplaybreaks

\usepackage{hyperref}
\hypersetup{colorlinks=true,citecolor=brown,linkcolor=brown,urlcolor=black,filecolor=black}

\makeatletter
    
    \@addtoreset{equation}{section}
\makeatother

\newcommand{\Z}{\mathbb{Z}}
\newcommand{\Q}{\mathbb{Q}}

\newcommand{\Qt}{\mathbb{Q}[t^{\pm1}]}

\newcommand{\Int}{\mathrm{Int}}
\newcommand{\SL}{\mathrm{SL}}
\newcommand{\lk}{\mathrm{lk}}

\newcommand{\R}{\mathcal{R}}
\newcommand{\A}{\mathcal{A}}
\renewcommand{\L}{\mathcal{L}}

\newcommand{\torus}{\mathbb{T}}
\newcommand{\diff}{\mathop{}\!\mathrm{d}}
\newcommand{\neigh}{\operatorname{N}}

\newcommand{\circleGdown}{
\begin{scope} [scale=0.2]
\draw[blue] (0,0) circle (3);
\draw[blue,->] (-3,0.1) -- (-3,0);
 \draw[blue] (0,-3) node{$\star$};
\end{scope}}

\newcommand{\circleLf}[2]{
\begin{scope} [scale=0.2,xshift=#2*4cm]
\draw[white,line width=10pt] (-2.46,1.72) arc (145:475:3);
\draw (-2.46,1.72) arc (145:475:3);
\draw (0,3) node[above] {$#1$};
\end{scope}}

\newcommand{\circleLup}[3]{
\begin{scope} [scale=0.2,xshift=#2*4cm]
\draw[white,line width=10pt] (-2.46,1.72) arc (145:475:3);
\draw[#1] (-2.46,1.72) arc (145:475:3);
\draw[#1,->] (3,-0.1) -- (3,0);
 \draw[#1] (0,-3) node{#3};
\end{scope}}

\newcommand{\circleLdown}[3]{
\begin{scope} [scale=0.2,xshift=#2*4cm]
\draw[white,line width=10pt] (-2.46,1.72) arc (145:475:3);
\draw[#1] (-2.46,1.72) arc (145:475:3);
\draw[#1,->] (3,0.1) -- (3,0);
\draw[#1] (0,-3) node{#3};
\end{scope}}

\newcommand{\circleRf}[2]{
\begin{scope} [scale=0.2,xshift=#2*4cm]
\draw[white,line width=10pt] (-1.27,-2.72) arc (-115:215:3);
\draw (-1.27,-2.72) arc (-115:215:3);
\draw (0,3) node[above] {$#1$};
\end{scope}}

\newcommand{\circleRup}[2]{
\begin{scope} [scale=0.2,xshift=#2*4cm]
\draw[white,line width=10pt] (-1.27,-2.72) arc (-115:215:3);
\draw[#1] (-1.27,-2.72) arc (-115:215:3);
\draw[#1,->] (3,-0.1) -- (3,0);
\draw[#1] (0,-3) node{$\bullet$};
\end{scope}}

\newcommand{\cups}[2]{
\begin{tikzpicture} [scale=0.2]
 \draw (0,0) node[above] {$\scriptscriptstyle{#1}$} arc (-180:0:1) node[above] {$\scriptscriptstyle{#2}$};
\end{tikzpicture}}

\newcommand{\caps}[2]{
\raisebox{-2ex}{
\begin{tikzpicture} [scale=0.2]
 \draw (0,0) node[below] {$\scriptscriptstyle{#1}$} arc (+180:0:1) node[below] {$\scriptscriptstyle{#2}$};
\end{tikzpicture}}}

\newcommand{\tetacup}[2]{
\hspace{-1ex}
\raisebox{-0.8ex}{
\begin{tikzpicture} [scale=0.2]
 \draw (0,0) node[above] {$\scriptscriptstyle{#1}$} arc (-180:-90:1);
 \draw (3,-1) arc (-90:0:1) node[above] {$\scriptscriptstyle{#2}$};
 \draw (2,-1) circle (1);
\end{tikzpicture}}}

\newcommand{\wheel}{{
\hspace{-3pt}
\raisebox{2pt}{
\begin{tikzpicture} [scale=0.07]
 \draw[very thin] (-1,1) -- (1,-1) (-1,-1) -- (1,1);
 \draw[very thin,fill=white] (0,0) circle (0.7);
\end{tikzpicture}}}}

\newcommand{\annulus}{{
\hspace{-3pt}
\raisebox{1pt}{
\begin{tikzpicture} [scale=0.09]
 \draw[very thick,fill=white] (0,0) circle (0.7);
\end{tikzpicture}}}}

\newcommand{\teta}{\raisebox{-0.5ex}{
\begin{tikzpicture} [scale=0.2]
 \draw (0,0) circle (1);
 \draw (-1,0) -- (1,0);
\end{tikzpicture}}}

\newcommand{\tetatwo}{\raisebox{-0.6ex}{
\begin{tikzpicture} [scale=0.16]
 \draw (1,0) -- (1,1) arc (0:180:1) -- (-1,0) arc (-180:0:1);
 \draw (-1,0) -- (1,0) (-1,1) -- (1,1);
\end{tikzpicture}}}

\newcommand{\tetafour}{\raisebox{-1ex}{
\begin{tikzpicture} [scale=0.2,rotate=45]
 \draw (0,0) circle (1);
 \draw (1,0) -- (2,0);
 \draw (0,1) -- (0,2);
 \draw (-1,0) -- (-2,0);
 \draw (0,-1) -- (0,-2);
\end{tikzpicture}}}

\newcommand{\tetalegs}{\raisebox{-0.6ex}{
\begin{tikzpicture} [scale=0.2]
 \draw (0,0) circle (1);
 \draw (-1,0) -- (1,0);
 \foreach \t in {60,120} {\draw[rotate=\t] (1,0) -- (2,0);}
\end{tikzpicture}}}

\newcommand{\tetaleg}[1]{\raisebox{-0.6ex}{
\begin{tikzpicture} [scale=0.2]
 \draw (0,0) circle (1);
 \draw (-1,0) -- (1,0);
 \draw[rotate=60] (1,0) -- (2,0) node[above] {$\scriptscriptstyle{#1}$};
 \draw[rotate=120] (1,0) -- (2,0) node[above] {$\scriptscriptstyle{#1}$};
\end{tikzpicture}}}

\newcommand{\omegafour}[1]{\raisebox{-4ex}{
\begin{tikzpicture} [scale=0.2,rotate=45]
 \draw (0,0) circle (1);
 \draw (1,0) -- (2,0) node[above] {$\scriptscriptstyle{#1}$};
 \draw (0,1) -- (0,2) node[above] {$\scriptscriptstyle{#1}$};
 \draw (-1,0) -- (-2,0) node[below] {$\scriptscriptstyle{#1}$};
 \draw (0,-1) -- (0,-2) node[below] {$\scriptscriptstyle{#1}$};
\end{tikzpicture}}}

\newcommand{\dfour}{\raisebox{-0.5ex}{
\begin{tikzpicture} [scale=0.3]
 \foreach \t in {0,120,240} \draw[rotate=\t] (0,0) -- (0,1) -- (0.87,-0.5);
\end{tikzpicture}}}

\newcommand{\figstrut}{\raisebox{-0.07cm}{
\begin{tikzpicture} [xscale=0.18,yscale=0.15]
 \draw (-2,0) -- (0,0);
 \draw[->,thick,blue] (1,1) arc (90:450:1);
\end{tikzpicture}}}

\definecolor{purple}{rgb}{.7,0,.7}
\definecolor{bleu}{rgb}{0,0,.7}

\title{A splicing formula for the LMO invariant}

\author{Gw\'ena\"el Massuyeau}
\address{Institut de Math\'ematiques de Bourgogne, UMR 5584, CNRS, 
Universit\'e Bourgogne Franche-Comt\'e, 21000 Dijon, France}
\email{{gwenael.massuyeau@u-bourgogne.fr}}

\author{Delphine Moussard}
\address{Institut de Math\'ematiques de Marseille, UMR 7373, Universit\'e d'Aix--Marseille, Marseille, France}
\email{delphine.moussard@univ-amu.fr}

\thanks{This research has been funded by the project ``ITIQ-3D'' of the R\'egion Bourgogne Franche--Comt\'e.
 G.M. is partly supported by the project ``AlMaRe'' (ANR-19-CE40-0001-01)
 and by the EIPHI Graduate School (ANR-17-EURE-0002).}

\begin{document}

\begin{abstract}
We prove a ``splicing formula'' for the LMO invariant, which is the universal finite-type invariant of rational homology $3$--spheres.
Specifically, if a rational homology $3$--sphere $M$ is obtained 
by gluing the exteriors of two framed knots $K_1 \subset M_1$ and $K_2\subset M_2$ in  rational homology $3$--spheres,
our formula expresses the LMO invariant of~$M$ in terms of the Kontsevich--LMO invariants of $(M_1,K_1)$ and $(M_2,K_2)$.
The proof uses the techniques that Bar-Natan and Lawrence 
developed to obtain a rational surgery formula {for} the LMO invariant.
In low degrees, we recover Fujita's formula for the Casson--Walker invariant
and we observe that the second term of the Ohtsuki series is not additive under ``standard'' splicing.
The splicing formula also works when each $M_i$ comes with a link $L_i$ in addition to the knot~$K_i$, hence we get a ``satellite formula'' for the Kontsevich--LMO invariant.
\end{abstract}

\maketitle

\tableofcontents

\vspace{-0.5cm}

\section{Introduction}

The LMO invariant  of closed oriented $3$--manifolds was constructed by Le, Murakami and Ohtsuki \cite{LMO}: from the Kontsevich integral of a surgery link, they derived a quantity which is invariant under Kirby moves. 
According to Le \cite{Le}, the LMO invariant of  integral homology $3$--spheres is universal among all $\Q$-valued finite-type invariants in the sense of \cite{Ohtsuki,GGP,Habiro}. 
A more general result is true for rational homology $3$--spheres whose first homology groups have a fixed cardinality \cite{Mou2,Massuyeau}.
In the case of a rational homology $3$--sphere $M$, the LMO invariant of $M$ coincides with its Aarhus integral $Z(M)$ \cite{Aarhus1,Aarhus2}.
The invariant $Z(M)$ takes values in a graded vector space of trivalent diagrams, which is denoted by $\A(\emptyset)$.
Its first non-trivial term, the coefficient of the $\theta$--shaped diagram, is given by the Casson--Walker invariant $\lambda_{\operatorname{W}}(M)$ as normalized in \cite{Walker}:
$$
Z(M) = \emptyset + \frac{\lambda_{\operatorname{W}}(M)}{4} \teta
+ \Big(\!\begin{array}{c}\hbox{\footnotesize trivalent diagrams} \\ \hbox{\footnotesize with $\geq 4$ vertices} \end{array}\!\Big)
 \ \in \A(\emptyset).
$$

The \emph{splicing operation} is the general procedure by which, given two framed knots in oriented $3$--manifolds, 
one creates a new $3$--manifold by gluing the exteriors of these knots.
The term ``splicing'' is sometimes used in the literature for what is called below the \emph{standard} splicing, namely the case when the gluing homeomorphism identifies the meridian of each of the two knots with the parallel of the other one. Standard splicing has the property to preserve the class of integral homology $3$--spheres (a knot in such a manifold being framed  with its ``preferred'' parallel). Another special case of splicing is given by the rational surgery on a knot. 

The main result of this paper expresses the LMO invariant of a rational homology $3$--sphere that is defined as the splice of two framed knots in rational homology $3$--spheres, in terms of the Kontsevich--LMO invariants of the two knots. For simplicity, we will first consider the case of null-homologous~knots (which are framed with the ``preferred'' parallel).

\begin{theo} 
 Let $K_1 \subset M_1$ and $K_2 \subset M_2$ be null-homologous knots in rational homology $3${--}spheres. Consider a splice $M$ of $(M_1,K_1)$ and $(M_2,K_2)$ that is also a rational homology $3$--sphere. Then, we have 
 \begin{equation} \label{eq:main}
 Z(M)= \omega\, \exp\left(\frac{1}{48}\left({-S\left(\frac pr\right)+\frac{p+s}r}\right)\! \teta \right) \, \left\langle \partial_{p,r} \big(Z^\wheel(K_1)\big) \,  ,\, \partial_{s,r} \big(Z^\wheel(K_2)\big) \right\rangle_r
 \end{equation}
 where $Z^\wheel(K_i)$ is the ``wheeled'' version of the Kontsevich--LMO invariant of $K_i\subset M_i$, 
 the integers $p,r,s$ depend on the gluing homeomorphism, 
 the term $\omega \in \A(\emptyset)$ is a  constant, $S(p/r)$ is a Dedekind symbol 
 and $ \partial_{p,r}, \partial_{s,r}, \langle -,-\rangle_r$ are  diagrammatic operations depending  
 only on the indicated integers $p,r,s$. 
\end{theo}

The reader is referred to Theorem \ref{thFormuleSplicing} for a precise statement of this ``splicing formula''. 
Although the LMO invariant does not separate rational homology $3$--spheres \cite{BL}, 
it is still unknown whether it separates integral homology $3$--spheres; 
nevertheless, Bar-Natan and Lawrence proved that it does separate Seifert fibered spaces that are integral homology $3$--spheres. 
They also computed the LMO invariant of all lens spaces. 
We expect that the splicing formula will be useful to compute (to some extent) the LMO invariant of new families of rational homology $3$--spheres.
We  particularly think of graph manifolds (within this homology type)
since they are obtained from Seifert fibered spaces by repeated
splicings along their fibers. 
A~possible outcome of this study would be to decide whether the LMO invariant also distinguishes 
graph manifolds that are integral homology $3$--spheres.

\

The paper is organized as follows.
In Section \ref{sec:spl_sur}, we describe the splice of two framed knots 
in closed oriented $3$--manifolds --- say $K_1 \subset M_1$ and $K_2 \subset M_2$ ---
as the result of the surgery along a framed link $K_1 \sqcup H \sqcup {K_2}$ in the connected sum $M_1\sharp M_2$, 
where $H$ is a ``chain'' of Hopf links ``clasping'' $K_1$ and $K_2$.
This surgery description of the splicing operation is well-known in, at least, two special cases:
for standard splicings and for rational surgeries.
Even in its most general form, the reader could deduce it from \cite[\S 5.2]{Gordon} by determining the handle decomposition of the plumbed $4$--manifold that is constructed there. 
Here the surgery description of a splice is proved by purely $3$--dimensional arguments.

In Section \ref{sec:Z(splice)}, we prove the above theorem.
Starting with the surgery description of a splice,
we use the techniques that Bar-Natan and Lawrence developed to produce a rational surgery formula for the LMO invariant \cite{BL};
their work relies itself on the ``Wheels and Wheeling'' conjectures proved in \cite{BLT}.
The rational surgery formula is reproved here as a consequence of the splicing formula (Corollary \ref{cor:BL}).

In Section \ref{sec:low_degree}, the splicing formula is made explicit in low degrees.
In degree two, we recover Fujita's splicing formula \cite{Fujita}
for the Casson--Walker invariant $\lambda_{\operatorname{W}}$, which involves the second derivatives of the Alexander polynomial of knots 
(Proposition \ref{PropFujita}).
In degree {four}, we obtain a splicing formula
for the second term $\lambda_{2}$ of the Ohtsuki series which, in addition to the second and fourth derivatives of
the Alexander polynomial, needs the third coefficient of the expansion of the Jones polynomial (Proposition \ref{PropDegFour}).
It turns out that, in contrast with~$\lambda_{\operatorname{W}}$, 
the invariant~$\lambda_{2}$ is not additive with respect to standard splicing.

The last two sections provide  more general versions of the splicing formula, 
which will be needed for the above-mentioned project of studying the LMO invariant of graph manifolds.
In Section~\ref{sec:links}, we consider the situation 
where each $M_i$  in the above theorem 
comes with an additional link~$L_i$ (disjoint from $K_i$).
Then, the Kontsevich--LMO invariant of $L_1 \sqcup L_2 \subset M$ is expressed by the same formula \eqref{eq:main} in terms of the Kontsevich--LMO invariants of $K_1\sqcup L_1 \subset M_1$ and $K_2 \sqcup L_2 \subset M_2$ (Theorem~\ref{th:links}).
As explained in the monograph  \cite{EN}, splicing is a fundamental operation in knot theory:
indeed, splicing subsumes all satellite operations (such as connected sum, cabling, Whitehead doubling, etc).
Thus, we give a ``satellite formula'' for the Kontsevich--LMO invariant (Corollary \ref{cor:satellite})
and we derive from this a result of Suetsugu~\cite{Suetsugu}.
Finally, we prove in Section \ref{sec:non-null} a generalization of the above theorem where each knot $K_i$ is allowed to be non-trivial in homology.
Then formula~\eqref{eq:main} extends by also taking into account the self-linking numbers of $K_1$ and $K_2$ (Theorem \ref{th:non-null}).

The paper ends with an appendix which collects useful facts about tridiagonal matrices and their signatures in terms of Dedekind sums.

\

\noindent
\textbf{Conventions.}
Throughout the paper, all manifolds are assumed to be compact, connected and oriented. 
The boundary of a manifold (if any) is oriented with the ``outward normal first'' convention. 

Unless otherwise stated, all knots are oriented. Given a framed knot $K$ in a 3--manifold $M$, the parallel $\rho(K)$ of $K$ (defining the framing) inherits from $K$ an orientation and the meridian $\mu(K)$ of $K$ is oriented so that $\lk(K,\mu(K))=+1$. 
Given a framed link $L$ in a 3--manifold $M$, the manifold obtained from $M$ by surgery on $L$ is denoted $M_L$.

\section{Splicing and surgery} \label{sec:spl_sur}

Let $M_1$ and $M_2$ be closed 3--manifolds, and 
let $K_1\subset M_1$ and $K_2\subset M_2$ be knots. 
For $i=1,2$, let $\neigh(K_i)$ be a tubular neighborhood of $K_i$ and set $X_i:=M_i\setminus\Int(\neigh(K_i))$.  
Given an orientation-reversing homeomorphism 
$$f:\partial\neigh(K_1)\fl{\scriptstyle{\cong}}\partial\neigh(K_2),$$ 
define the {\em splice} of $(M_1,K_1)$ and $(M_2,K_2)$ along $f$ as the closed 3--manifold 
\begin{equation} \label{eq:M}
M:= X_1\, \mathop{\bigcup}_{f}\, X_2 .
\end{equation}
The goal of this section is to describe $M$ as a surgery along a framed link in the connected sum $M_1\sharp M_2$. 

\begin{remark}
There is a more general ``self-splicing'' operation, which is defined as follows: 
given a closed (possibly disconnected) $3$--manifold $M_0$, two disjoint knots
$K_1,K_2 \subset M_0$, and an orientation-reversing homeomorphism 
$f:\partial\neigh(K_1)\to \partial\neigh(K_2),$
the \emph{self-splice} $M$ of $(M_0,K_1,K_2)$ is the result of self-gluing $M_0\setminus \Int(\neigh(K_1) \cup \neigh(K_2))$ using $f$.
Let $M_i$ be the connected component of $M_0$ containing $K_i$ for $i=1,2$, and assume for simplicity that $M_0=M_1 \cup M_2$.
If $M_1\neq M_2$, then $M$ is the splice of $(M_1,K_1)$ and $(M_2,K_2)$.
If $M_1 =M_2$, then $M$ is not a rational homology $3$--sphere even if $M_0$ is assumed to be so.
Since we are only interested in rational homology $3$--spheres in this paper, we will not further consider self-splices.
\end{remark}

Fix a ``model'' $\torus$ of the $2$--dimensional torus, together with a basis $(\alpha,\beta)$ of its fundamental group. The surface $\torus$ is oriented so that the intersection number $\alpha \cdot \beta$ is $+1$.
If some identifications $\torus\cong - \partial\neigh(K_1)$ and $\torus\cong\partial\neigh(K_2)$ are fixed, 
then the splice \eqref{eq:M} can be achieved by gluing the mapping cylinder of an orientation-preserving self-homeomorphism  of $\torus$ to $X_1\sqcup X_2$. We will use this point of view to get a surgery description of the splicing operation.

To this aim, we study {\em toroidal cobordisms}, namely 3--manifolds $C$ whose boundary $\partial C= (-\partial_-C)\sqcup \partial_+C $ is decomposed into two parts, $\partial_+ C$ and $\partial_- C$, each being identified to  $\torus$.
The \emph{composition} $D \circ C$ of two toroidal cobordisms, $C$ and $D$, is defined by gluing
$C$ to $D$ using the given identifications $\partial_+C \cong \torus \cong \partial_- D$.
There is a one-to-one correspondence between pairs $(M,J)$, where $J$ is a $2$--component framed link in a closed $3$--manifold $M$,
and toroidal cobordisms $C$. 
Specifically, the cobordism corresponding to a $2$--component framed link $J=J'\sqcup J''$ in a closed 3--manifold $M$ is its exterior 
$$
C := M\setminus\left(\Int(\neigh(J'))\sqcup\Int(\neigh(J''))\right)
$$ 
where $\partial_-C :=  \partial\neigh(J')$ is identified to $\torus$ by $(\mu(J'), \rho(J')) \mapsto (\alpha,\beta) $ and $\partial_+ C :=  - \partial\neigh(J'')$ is identified to $\torus$ by $(\mu(J''), \rho(J'')) \mapsto (\alpha,-\beta)$.
Here we use the conventions that have been set at the very end of the introduction.

In this way, toroidal cobordisms can be presented by a surgery link in the complement of the trivial $2$--component framed link in $S^3$.
For instance, Figure \ref{figExCob} represents $M$ as the 3--manifold obtained by surgery on the black framed unoriented link, while $J'$ and $J''$ are the copies in $M$ of the blue (\textcolor{blue}{$\star$}--marked) knot and the red (\textcolor{red}{$\bullet$}--marked) knot, respectively; each $k\in \Z$ decorating a black (unmarked) component is here to mean that we are performing a surgery with a framing number that differs by $k$ from the ``blackboard framing''. 
In the sequel, we shall use these diagrammatic conventions  for surgery presentations.
\begin{figure}[htb] 
\begin{center}
\begin{tikzpicture} [scale=0.6]
\newcommand{\extg}[1]{\draw[rotate=#1] (0,-1) .. controls +(1.2,-0.9) and +(1.3,-0.7) .. (0.97,0.45);}
\newcommand{\inted}[1]{\draw[rotate=#1] (-0.84,0.38) .. controls +(0.15,-0.6) and +(-0.5,0.5) .. (0,-1);}
\extg{0} \extg{120} \extg{240} \inted{0} \inted{120} \inted{240} 
\draw (-1.2,1) node {$3$};
\begin{scope} [xshift=2.5cm,yshift=-2cm,scale=0.33]
\draw (0.2,0.2) .. controls +(1,1) and +(1,-1) .. (0,5) .. controls +(-1,1) and +(0,-1) .. (-2,7.7) 
  (-2,8.3) .. controls +(0,2) and +(0,3) .. (2,8) .. controls +(0,-1) and +(0.7,1) .. (0.1,5.2)
  (-0.15,4.85) .. controls +(-0.7,-1) and +(-1,1) .. (0,0) .. controls +(4,-4) and +(4,0) .. (2.3,8) 
  (1.7,8) -- (-2,8);
  \draw[white,line width=6pt] (-2,8) .. controls +(-2,0) and +(-0.5,2) .. (-4,3);
  \draw (-2,8) .. controls +(-2,0) and +(-0.5,2) .. (-4,3) (-3.9,2.4) .. controls +(0.5,-2) and +(-2,-2) .. (-0.2,-0.2);
  \draw (3,10) node {$-2$};
\end{scope}
\draw[white,line width=7pt] (-1.7,-0.1) arc (70:410:1);
\draw[blue] (-1.7,-0.1) arc (70:410:1);
\draw[blue,->] (-2,-0.05) -- (-2.1,-0.05);
\draw[blue] (-2,-2.05) node{$\star$};
\draw[white,line width=7pt] (3.7,0) arc (110:450:1);
\draw[red] (3.7,0) arc (110:450:1);
\draw[red,->] (5.05,-1) -- (5.05,-0.9);
\draw[red] (4.1,-1.95) node{$\bullet$};
\end{tikzpicture}
\end{center}
\caption{Surgery presentation of a toroidal cobordism} \label{figExCob}
\end{figure}
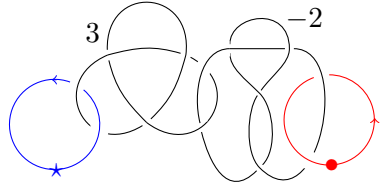

\begin{lemma} \label{lemmaTrivialCob}
 The trivial toroidal cobordism is $I=$\raisebox{-0.4cm}{
\begin{tikzpicture} [scale=0.5]
 \draw[blue] (0,0) circle (1);
 \draw[blue,->] (-1,0.1) -- (-1,0);
 \draw[blue] (0,-1) node{$\star$};
 \draw[red] (4,0) circle (1);
 \draw[red,->] (5,-0.1) -- (5,0);
 \draw[red] (4,-1) node{$\bullet$};
 \draw[white,line width=7pt] (0.8,-0.3) -- (3.2,-0.3);
 \draw (0.8,0.3) arc (90:270:0.3) -- (3.2,-0.3) arc (-90:90:0.3);
 \draw (1.2,0.3) -- (2.8,0.3);
 \draw (2,0.3) node[above] {0};
\end{tikzpicture}} . 
\end{lemma}
\begin{proof}
 The manifold obtained by surgery is $S^2\times S^1$; the knots $J'$ and $J''$ that correspond to the blue (\textcolor{blue}{$\star$}--marked) component and to the red (\textcolor{red}{$\bullet$}--marked) component, respectively, are of the form $\{*\}\times S^1$. Hence the cobordism is an annulus times $S^1$, namely a thickened torus. 
 Moreover, the meridian $\mu(J')$ can be slid over the disk glued by surgery to get $\mu(J'')$. 
 Finally, $\rho(J')$ is clearly homotopic to $-\rho(J'')$.
\end{proof}

\begin{lemma} \label{lemmaCompCob}
 Let $C$ and $D$ be toroidal cobordisms, which are respectively associated to pairs $(M,J)$ and $(N,L)$ of framed links in 3--manifolds. 
 Then the composed cobordism $D \circ C$ is associated to the pair
 $\big((M\sharp N)_{J''\sharp L'},J'\sqcup L'' \big).$
\end{lemma}

\begin{proof}
 Perform the connected-sum inside the tubular neighborhoods $\neigh(J'')$ and $\neigh(L')$. 
 Then the cobordism associated to $\big((M\sharp N)_{J''\sharp L'},J'\sqcup L''\big)$ is 
 $$C'\circ \big(\neigh(J'')\, \sharp\, \neigh(L')\big)_{J''\sharp L'}\circ C.$$ 
 It is easily checked that $\big(\neigh(J'')\sharp\neigh(L')\big)_{J''\sharp L'}$ is the trivial cobordism
 (e.g., using Lemma \ref{lemmaTrivialCob}).
\end{proof}

For any $a \in \Z$, we introduce the following toroidal cobordisms: 
\begin{equation} \label{eq:RL}
\R_a=\raisebox{-4ex}{
      \begin{tikzpicture}
       \circleGdown
       \circleRf{0}{1}
       \circleRf{a}{2}
       \circleRup{red}{3}
      \end{tikzpicture}}
  \qquad
  \L_a=\raisebox{-4ex}{
      \begin{tikzpicture}
       \circleGdown
       \circleLf{a}{1}
       \circleLf{0}{2}
       \circleLup{red}{3}{$\bullet$}
      \end{tikzpicture}}
\end{equation}
Let us compute the composition $\L_{-a}\circ \R_{a}$ using Lemma \ref{lemmaCompCob}: 
\begin{align*}
 \L_{-a}\circ \R_{a} &=\raisebox{-4ex}{
 \begin{tikzpicture}
  \circleGdown
  \circleRf 0 1
  \circleRf a 2
  \circleRf 0 3
  \circleLf{-a}4
  \circleLf 0 5
  \circleLup{red}{6}{$\bullet$}
 \end{tikzpicture}}\\
 &=\raisebox{-4ex}{
 \begin{tikzpicture}
  \circleGdown
  \circleRf 0 1
  \circleRf 0 2
  \circleLf 0 3
  \circleLup{red}{4}{$\bullet$}
 \end{tikzpicture}} 
 \ = \ \raisebox{-4ex}{
 \begin{tikzpicture}
  \circleGdown
  \circleRf 0 1
  \circleLup{red}{2}{$\bullet$}
 \end{tikzpicture}} \ = \ I
\end{align*}
The second and third equalities are given by slam-dunk moves as represented in the right hand side of Figure~\ref{figSlamDunk} (see \cite{CG88} for the left hand side and deduce the right hand side by sliding first the rightmost component on the leftmost one). 
We deduce that the cobordisms $\R_a$ and $\L_a$ are invertible, which implies that they are mapping cylinders. We focus on the $\L_a$ and compute their action in homology.
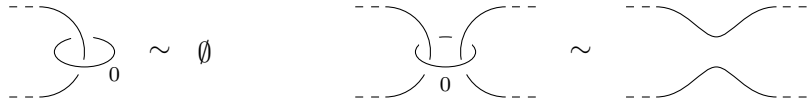
\begin{figure}[htb]
\begin{center}
\begin{tikzpicture} [scale=0.2]
 \draw[dashed] (-5,3) -- (-3,3) (-5,-3) -- (-3,-3);
 \draw (0,0) ellipse (2 and 1);
 \draw[white,line width=8pt] (-3,3) arc (90:-10:3);
 \draw (-3,3) arc (90:-10:3);
 \draw (-3,-3) arc (-90:-30:3);
 \draw (2,-1.5) node {$\scriptstyle{0}$} (5,0) node {$\sim$} (8,0) node {$\emptyset$};
\begin{scope} [xshift=24cm]
 \draw[dashed] (-6,3) -- (-4,3) (-6,-3) -- (-4,-3);
 \draw[dashed] (6,3) -- (4,3) (6,-3) -- (4,-3);
 \draw (0,0) ellipse (2 and 1);
 \draw[white,line width=8pt] (-4,3) arc (90:-10:3);
 \draw (-4,3) arc (90:-10:3);
 \draw (-4,-3) arc (-90:-30:3);
 \draw[white,line width=8pt] (4,3) arc (90:190:3);
 \draw (4,3) arc (90:190:3);
 \draw (4,-3) arc (-90:-150:3);
 \draw (0,-1) node[below] {$\scriptstyle{0}$};
\end{scope}
 \draw (33,0) node {$\sim$};
\begin{scope} [xshift=42cm]
 \draw[dashed] (-6,3) -- (-4,3) (-6,-3) -- (-4,-3);
 \draw[dashed] (6,3) -- (4,3) (6,-3) -- (4,-3);
 \draw (-4,3) .. controls +(2,0) and +(-1,0) .. (0,1) .. controls +(1,0) and +(-2,0) .. (4,3);
 \draw (-4,-3) .. controls +(2,0) and +(-1,0) .. (0,-1) .. controls +(1,0) and +(-2,0) .. (4,-3);
\end{scope}
\end{tikzpicture} \caption{A slam-dunk move and a corollary} \label{figSlamDunk}
\end{center}
\end{figure}

\begin{lemma} \label{lemmaCobActionHom}
 The action of the cobordism $\L_a$ on the homology of $\torus$ is given by the matrix $\begin{pmatrix} a & -1 \\ 1 & 0 \end{pmatrix}$ in the basis $(\alpha,\beta)$ of $H_1(\torus;\Z)$.
\end{lemma}
\begin{proof}
 Denote by $L_a$ and $L_0$ the surgery components of $\L_a$ in \eqref{eq:RL}
 framed by $a$ and~$0$ respectively, and orient each of them in the counter-clockwise direction. 
 We have the following identities in $H_1(\L_a;\Z)$: 
 \begin{align*}
  \rho(J')=\mu(L_a) \qquad \rho(L_a)=\mu(J')+\mu(L_0) & \qquad \rho(L_a)+a\,\mu(L_a)=0 \\
  \rho(J'')= \mu(L_0) \qquad \rho(L_0)=\mu(L_a)+\mu(J'') & \qquad \rho(L_0)=0
 \end{align*}
 This gives $\mu(J')=a\,\mu(J'')-\rho(J'')$ and $\rho(J')=-\mu(J'')$, and we can conclude. 
\end{proof}

\begin{lemma} \label{lem:Hopf_chain} 
 For any integers $a_1,\dots,a_n$, we have
 $$\L_0\circ \L_{a_n}\circ\dots\circ \L_{a_1}=\left\lbrace\begin{array}{ll}
        \raisebox{-4ex}{
        \begin{tikzpicture} [scale=0.8]
         \circleGdown
         \circleLf{a_1}1
         \circleLf{}2
         \circleLf{}5
         \circleLf{a_{n-1}}6
         \circleLf{a_n}7
         \circleLdown{red}{8}{$\bullet$}
         \draw[white,fill=white] (1.6,-0.62) -- (4,-0.62) -- (4,0.62) -- (1.6,0.62) -- (1.6,-0.62);
         \draw (2.8,0) node {$\dots$};
        \end{tikzpicture}}& \textrm{if }n\textrm{ is even,} \\
        \raisebox{-4ex}{
        \begin{tikzpicture} [scale=0.8]
         \circleGdown
         \circleLf{a_1}1
         \circleLf{}2
         \circleLf{}5
         \circleLf{a_{n-1}}6
         \circleLf{a_n}7
         \circleLup{red}{8}{$\bullet$}
         \draw[white,fill=white] (1.6,-0.62) -- (4,-0.62) -- (4,0.62) -- (1.6,0.62) -- (1.6,-0.62);
         \draw (2.8,0) node {$\dots$};
        \end{tikzpicture}}& \textrm{if }n\textrm{ is odd.} 
 \end{array}\right.$$
\end{lemma}
\begin{proof}
 For any integers $a$ and $b$, we have 
 \begin{align*}
 \L_a\circ\L_b & =\raisebox{-3ex}{
 \begin{tikzpicture} [scale=0.8]
   \circleGdown
   \circleLf b 1
   \circleLf 0 2
   \circleLf 0 3
   \circleLf a 4
   \circleLf 0 5
   \circleLup{red}{6}{$\bullet$}
 \end{tikzpicture}}\\
 &=\raisebox{-3ex}{
 \begin{tikzpicture} [scale=0.8]
   \circleGdown
   \circleLf b 1
   \circleLf 0 2
   \circleLf 0 3
   \circleRf a 4
   \circleLf 0 5
   \circleLdown{red}{6}{$\bullet$}
 \end{tikzpicture}} \ = \ \raisebox{-3ex}{
 \begin{tikzpicture} [scale=0.8]
   \circleGdown
   \circleLf b 1
   \circleLf a 2
   \circleLf 0 3
   \circleLdown{red}{4}{$\bullet$}
 \end{tikzpicture}}
 \end{align*}
 where the first equality is given by Lemma \ref{lemmaCompCob}, the second one by an isotopy 
 and the third one by slam-dunk. 
 For any integers $a_1,\dots,a_n, a_{n+1}$, a similar computation gives by induction:
 $$\L_{a_{n+1}}\circ \L_{a_{n}}\circ\dots\circ \L_{a_1}=\left\lbrace\begin{array}{ll}
        \raisebox{-4ex}{
        \begin{tikzpicture} [scale=0.8]
         \circleGdown
         \circleLf{a_1}1
         \circleLf{}2
         \circleLf{}5
         \circleLf{a_{n+1}}6
         \circleLf 0 7
         \circleLdown{red}{8}{$\bullet$}
         \draw[white,fill=white] (1.6,-0.62) -- (4,-0.62) -- (4,0.62) -- (1.6,0.62) -- (1.6,-0.62);
         \draw (2.8,0) node {$\dots$};
        \end{tikzpicture}}& \textrm{if }n\textrm{ is odd.} \\
        \raisebox{-4ex}{
        \begin{tikzpicture} [scale=0.8]
         \circleGdown
         \circleLf{a_1}1
         \circleLf{}2
         \circleLf{}5
         \circleLf{a_{n+1}}6
         \circleLf 0 7
         \circleLup{red}{8}{$\bullet$}
         \draw[white,fill=white] (1.6,-0.62) -- (4,-0.62) -- (4,0.62) -- (1.6,0.62) -- (1.6,-0.62);
         \draw (2.8,0) node {$\dots$};
        \end{tikzpicture}}& \textrm{if }n\textrm{ is even.} 
 \end{array}\right.$$
 Setting $a_{n+1}:=0$ and applying a slam-dunk move leads to the result.
\end{proof}

We now come back to the splice $M$ of $(M_1,K_1)$ and $(M_2,K_2)$ defined by an orientation-reversing homeomorphism $f:\partial\neigh(K_1) \to \partial\neigh(K_2)$ and we assume  that $K_1,K_2$ are framed. Using the framing, we make the following identifications:
\begin{eqnarray} \label{eq:tori}
-\partial X_2 = \partial\neigh(K_2) \stackrel{\cong}{\longrightarrow} \torus, &&\quad (\mu(K_2), \rho(K_2)) \longmapsto  (\alpha,\beta) \\
\notag \partial X_1= - \partial\neigh(K_1) \stackrel{\cong}{\longrightarrow} \torus,&& \quad (\mu(K_1), \rho(K_1))  \longmapsto  (\alpha,-\beta)
\end{eqnarray}
Thanks to these identifications, we view $f$ as an orientation-preserving self-homeomorphism of~$\torus$,
and we consider the matrix  giving the action of $f$ on $H_1(\torus;\Z)$:
\begin{equation} \label{eq:matrix_f}
\begin{pmatrix} p & r \\ q & s \end{pmatrix}
\quad \hbox{where} \quad 
\left\{\begin{array}{l} f_*(\alpha) = p\, \alpha + q\, \beta\\
f_*(\beta) = r\, \alpha + s\, \beta
\end{array}\right.
\end{equation} 
The following lemma gives a decomposition of this matrix. See \cite[Lemma~4]{Gordon} for a proof. 

\begin{lemma} \label{lemmaSL}
 Any matrix in $\SL_2(\Z)$ is a product of matrices $\begin{pmatrix} a & -1 \\ 1 & 0 \end{pmatrix}$ with $a\in\Z$.
\end{lemma}

%
%

Applying Lemma \ref{lemmaSL} to $\begin{pmatrix} 0 & 1 \\ -1 & 0 \end{pmatrix}\begin{pmatrix} p & r \\ q & s \end{pmatrix}$, we get some integers $a_1,\dots,a_n$ such that
\begin{equation} \label{eq:decomposition_f}
\begin{pmatrix} p & r \\ q & s \end{pmatrix}=
  \begin{pmatrix} 0 & -1 \\ 1 & 0 \end{pmatrix}
  \begin{pmatrix} a_n & -1 \\ 1 & 0 \end{pmatrix} \cdots
  \begin{pmatrix} a_1 & -1 \\ 1 & 0 \end{pmatrix}.
\end{equation}

\begin{proposition} \label{prop:splice}
 The splice $M$ of $(M_1,K_1)$ and $(M_2,K_2)$ defined by a homeomorphism $f$, whose matrix \eqref{eq:matrix_f} in homology
 is decomposed as in \eqref{eq:decomposition_f}, is the surgered manifold
 $$ (M_1\sharp M_2)_{K_1\sqcup H\sqcup K_2}$$ where $H:=H(a_1,\dots,a_n)$ is the ``chain'' of Hopf links ``clasping'' $K_1$ and $K_2$ as shown below: 
 $$\left\lbrace\begin{array}{ll}
        \raisebox{-6ex}
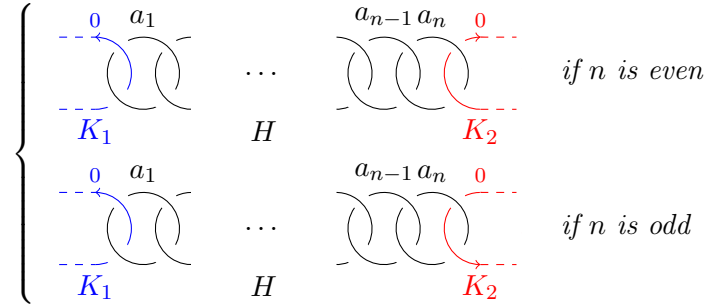
{
        \begin{tikzpicture} [scale=0.8]
         \draw[blue,dashed,scale=0.2] (-3,3) -- (0,3) (-3,-3) -- (0,-3);
         \draw[blue,scale=0.2] (0,-3) node[below] {$K_1$};
         \draw[blue,->,scale=0.2] (0.1,3) -- (0,3) node[above] {$\scriptstyle{0}$};
         \draw[blue,scale=0.2] (0,-3) arc (-90:90:3);
         \circleLf{a_1}1
         \circleLf{}2
         \circleLf{}5
         \circleLf{a_{n-1}}6
         \circleLf{a_n}7
         \circleLdown{red}{8}{}
         \draw[white,fill=white] (1.6,-0.62) -- (4,-0.62) -- (4,0.62) -- (1.6,0.62) -- (1.6,-0.62);
         \draw (2.8,0) node {$\dots$};
         \draw[white,fill=white,scale=0.2] (32,-3) -- (35.4,-3) -- (35.4,3) -- (32,3) -- (32,-3);
         \draw[red,scale=0.2,dashed] (32,-3) -- (35,-3) (32,3) -- (35,3);
         \draw[red,->,scale=0.2] (31.9,3) -- (32,3) node[above] {$\scriptstyle{0}$};
         \draw[red,scale=0.2] (32,-3) node[below] {$K_2$};
         \draw (2.8,-0.62) node[below] {$H$};
        \end{tikzpicture}}& \textrm{if }n\textrm{ is even} \\
        \raisebox{-6ex}{
        \begin{tikzpicture} [scale=0.8]
         \draw[blue,dashed,scale=0.2] (-3,3) -- (0,3) (-3,-3) -- (0,-3);
         \draw[blue,scale=0.2] (0,-3) node[below] {$K_1$};
         \draw[blue,->,scale=0.2] (0.1,3) -- (0,3) node[above] {$\scriptstyle{0}$};
         \draw[blue,scale=0.2] (0,-3) arc (-90:90:3);
         \circleLf{a_1}1
         \circleLf{}2
         \circleLf{}5
         \circleLf{a_{n-1}}6
         \circleLf{a_n}7
         \circleLup{red}{8}{}
         \draw[white,fill=white] (1.6,-0.62) -- (4,-0.62) -- (4,0.62) -- (1.6,0.62) -- (1.6,-0.62);
         \draw (2.8,0) node {$\dots$};
         \draw[white,fill=white,scale=0.2] (32,-3) -- (35.4,-3) -- (35.4,3) -- (32,3) -- (32,-3);
         \draw[red,scale=0.2,dashed] (32,-3) -- (35,-3) (32,3) -- (35,3);
         \draw[red,->,scale=0.2] (31.9,-3) -- (32,-3);
         \draw[red,scale=0.2] (32,3) node[above] {$\scriptstyle{0}$};
         \draw[red,scale=0.2] (32,-3) node[below] {$K_2$};
         \draw (2.8,-0.62) node[below] {$H$};
        \end{tikzpicture}}& \textrm{if }n\textrm{ is odd} 
 \end{array}\right..$$
\end{proposition}
\begin{proof}
 Let $C_f$ be the mapping cylinder of $f$, which we view as a toroidal cobordism.
 Since $C_f$ is determined by its action on homology, Lemma \ref{lemmaCobActionHom} implies 
 $$
 C_f=\L_0\circ\L_{a_n}\circ\dots\circ\L_{a_1}.
 $$ 
 Viewing $M$ as a cobordism $\emptyset \to \emptyset$, $X_1$ as a cobordism $\emptyset \to \torus$ 
 and $X_2$ as a cobordism $\torus \to \emptyset$, we have $M=X_2\circ C_f\circ X_1$. 
 Hence we conclude with Lemma \ref{lemmaCompCob} and Lemma~\ref{lem:Hopf_chain}.
\end{proof}

\section{Splicing formula} \label{sec:Z(splice)}

In the sequel, we mainly follow the notations of \cite{BL} which we recall in part. 
Let $N$ be a rational homology $3$--sphere (in short, a \emph{$\Q$HS})
and let $L\subset N$ be a framed knot.
We denote~by
$$
Z(N,L) \in \A(\circlearrowleft_l)
$$
the Kontsevich--LMO invariant of the pair $(N,L)$. 
Here the ``abstract'' oriented $1$--manifold~$\circlearrowleft$ is labeled with the letter $l$
in order to refer to the unique connected component of $L$: this convention (lower-to-upper case letters)
will be used throughout the text to label connected components.
In the absence of knot, $Z(N) := Z(N,\emptyset)$ is the LMO invariant of $N$ which is denoted by
$\hat Z^{\operatorname{LMO}}(N)  \in \A(\emptyset)$ in \cite{BL}\footnote{This corresponds to the notation $\hat \Omega(M)$ in \cite{LMO}
and to $\ring{A}(M)$ in \cite{Aarhus1}.}.
In the case of the standard $3$--sphere,  $Z(L) := Z(S^3,L)$ is the Kontsevich integral of $L$
in the version that is also denoted by $Z(L)\in \A(\circlearrowleft_l)$ in \cite{BL}.
Even if $N \neq S^3$, $Z(N,L)$ will often be abbreviated to $Z(L)$ when $N$ is clear from the context.

Let $\chi: \A(*_X) \to \A(\uparrow_X)$ be the diagrammatic analogue of the PBW isomorphism, which is defined for any finite set $X$. 
Denote by
$$
\Omega:= \chi^{-1}Z(\operatorname{unknot}) \in \A(*)
$$
the ``symmetrized'' value of the unknot.
The latter has been found in \cite{BLT} to be equal to
\begin{equation} \label{eq:Omega}
\Omega = \exp\Big(\sum_{m\geq 1} b_{2m}\, \omega_{2m} \Big)
\end{equation}
where $\omega_{2m}$ denotes the Jacobi diagram consisting of one ``wheel''  with $2m$ ``spokes'' (as depicted in Figure~\ref{figomega}), and the {\em modified Bernoulli numbers}
\begin{equation} \label{valuesb}
 b_2=\frac 1{48}, \quad b_4=-\frac 1{5760}, \quad \hbox{etc,}
\end{equation}
are defined by the formal power series
$$ 
\sum_{m\geq 1} b_{2m}X^{2m}: = \frac{1}{2} \log\Big(\frac{\sinh(X/2)}{X/2} \Big) \in \Q[[X]].
$$
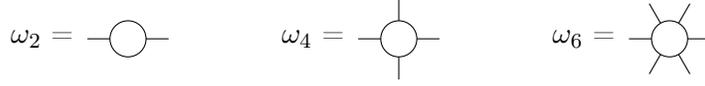
\begin{figure} [htb]
\begin{center}
\begin{tikzpicture} [scale=0.3]
 \draw (-3.8,0) node {$\omega_2$\ =};
 \foreach \t in {0,180} {\draw [rotate=\t] (0.8,0) -- (1.8,0);}
 \draw (0,0) circle (0.8);
\begin{scope} [xshift=12cm]
 \draw (-3.8,0) node {$\omega_4$\ =};
 \foreach \t in {0,90,180,270} {\draw [rotate=\t] (0.8,0) -- (1.8,0);}
 \draw (0,0) circle (0.8);
\end{scope}
\begin{scope} [xshift=24cm]
 \draw (-3.8,0) node {$\omega_6$\ =};
 \foreach \t in {0,60,...,300} {\draw [rotate=\t] (0.8,0) -- (1.8,0);}
 \draw (0,0) circle (0.8);
\end{scope}
\end{tikzpicture} \caption{The diagrams $\omega_{2m}$} \label{figomega}
\end{center}
\end{figure}

Given a finite set $X$ and $D\in \A(*_X)$, we denote by $\partial_D$ (resp$.$ $\langle D,- \rangle$) the operator that  maps any Jacobi diagram $E$ (having univalent vertices colored by $X$) to the sum of all ways of gluing
all the $x$--colored vertices of $D$ to  some of (resp$.$ to all) the $x$--colored vertices of $E$, for each $x\in X$.
In particular, we will need the constant
$$ 
\omega := \langle \Omega , \Omega \rangle  = \emptyset + \frac{1}{24} \tetatwo + 
\Big(\!\begin{array}{c}\hbox{\footnotesize trivalent diagrams} \\ \hbox{\footnotesize with $\geq 6$ vertices} \end{array}\!\Big)
\in \A(\emptyset).
$$
It has been proved in \cite{BLT} that the composition of linear isomorphisms
$$
\xymatrix{
\A(*) \ar[r]^-{\partial_{\Omega}}_-\cong & \A(*)   \ar[r]^-{\chi}_-\cong
& \A(\uparrow) \cong \A(\circlearrowleft) 
}
$$
preserves the algebra structures. 
Thus, as in \cite{BL}, we will work with the \emph{wheeled} version of the Kontsevich--LMO invariant, namely
$$
Z^\wheel(N,L) := \partial_{\Omega_l}^{-1} \chi^{-1}(Z(N,L)) \in \A(*_l).
$$
Here $\Omega_l\in \A(*_l)$ is the same as $\Omega \in \A(*)$ but with each univalent vertex now labeled with $l$.
For instance, \cite[Lemma 3.10]{BL} tells us that
\begin{equation} \label{eq:wheeled_Omega}
Z^\wheel(\operatorname{unknot})= \omega^{-1} \Omega \in \A(*).
\end{equation}

Let  now $M_1$ and $M_2$ be $\Q$HS. 
Let $K_1 \subset M_1$ and $K_2 \subset M_2$ be null-homologous knots. 
We give $K_i$ the \emph{preferred} parallel $\rho(K_i)$, 
i.e$.$ the one that bounds a surface in the knot exterior $X_i= M_i\setminus\Int(\neigh(K_i))$.
We identify $\partial X_i$ with $\torus$ as in \eqref{eq:tori},
and we consider the splice 
\begin{equation} \label{eq:M_bis}
M:= X_1\, \mathop{\bigcup}_{f}\, X_2
\end{equation}
defined by the homeomorphism $f:\torus \to \torus$
that is encoded by four integers $p,q,r,s$ as in \eqref{eq:matrix_f}.
A Mayer--Vietoris argument shows that
$M$ is a $\Q$HS  if and only if $r \neq0$ (see the proof of Lemma \ref{lem:MV} below).



Our main result computes the LMO invariant $Z(M)$ in terms of the wheeled Kontsevich--LMO invariants
$Z^\wheel(K_i) = Z^\wheel(M_i,K_i)$. 
The formula involves Dedekind symbols whose definition is recalled in Appendix \ref{app:tri}.
We also denote by $\theta$ the $\theta$--shaped diagram $\teta \in \A(\emptyset)$.

\begin{theorem} \label{thFormuleSplicing}
 Let $K_1 \subset M_1$ and $K_2 \subset M_2$ be null-homologous knots in $\Q$HS, and let $M$ be the splice of $(M_1,K_1)$ and $(M_2,K_2)$ as described by \eqref{eq:M_bis} in terms of four integers $p,q,r,s$. 
 If~$M$ is also a $\Q$HS,  then we have 
 $$
 Z(M)= \omega\, \exp\left(\frac{\theta}{48}\left({-S\left(\frac pr\right)+\frac{p+s}r}\right)\right) \, \left\langle \partial_{D_1}\big(Z^\wheel(K_1)\big)\big\vert_{k_1\to {-k_{2}}/{r}}\,  ,\, \partial_{D_2}\big(Z^\wheel(K_2)\big) \right\rangle
 $$
 where 
 \begin{equation} \label{eq:DD} 
 D_1 := \exp\Bigg(-\frac{p}{2r}\cups{k_1}{k_1}\Bigg) \quad \hbox{and} \quad 
 D_2 := \exp\Bigg(-\frac{s}{2r}\cups{k_2}{k_2}\Bigg).
 \end{equation}
\end{theorem}

\begin{proof}
Choose a matrix decomposition as in \eqref{eq:decomposition_f}, which leads to $n\geq 1$ integers $a_1,\dots, a_n$,
and consider the Hopf chain $\widetilde{H} := H(0,a_1,\dots,a_n,0) \subset S^3$ shown below:
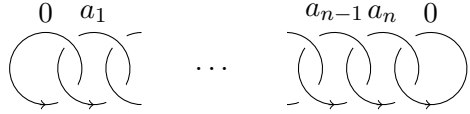
\begin{center}
\begin{tikzpicture} [scale=0.8]
 \draw (0,0) circle (0.6);
 \draw (0,0.6) node[above] {$0$};
 \circleLf{a_1}1
 \circleLf{}2
 \circleLf{}5
 \circleLf{a_{n-1}}6
 \circleLf{a_n}7
 \circleLf{0}8
 \draw[white,fill=white] (1.6,-0.62) -- (4,-0.62) -- (4,0.62) -- (1.6,0.62) -- (1.6,-0.62);
 \draw (2.8,0) node {$\dots$};
 \foreach \x in {0,1,6,7,8}{
 \draw[->] (0.8*\x-0.1,-0.6) -- (0.8*\x,-0.6);}
\end{tikzpicture}
\end{center} 
Reading from left to right, the  components of $\widetilde{H}$ are denoted $H', H_1,\dots , H_n , H''$.

Let $\epsilon_n K_2$ be the knot $K_2$ if $n$ is odd, and the knot $K_2$ {\emph{with reversed orientation}} if $n$ is even.
According to Proposition \ref{prop:splice}, $M$ is the result of doing surgery 
in $M_1 \sharp M_2$ along the link $K_1 \sqcup H \sqcup \epsilon_n K_2$ 
where $H$ is the Hopf chain $H(a_1,\dots,a_n)$ ``clasping'' positively $K_1$ on one side and $\epsilon_n K_2$ on the other side.
This link in $M_1 \sharp M_2 = (M_1 \sharp S^3)\sharp M_2$
is obtained from $(K_1 \sqcup \widetilde {H}) \sqcup K_2$ by doing the connected sum of $K_1\subset M_1$ with $H'\subset S^3$, in a first time, 
and by doing the connected sum of $H'' \subset M_1 \sharp S^3$ with $\epsilon_n K_2\subset M_2$, in a second time.
The way  the wheeled Kontsevich integral
behaves under connected sums of links is described in \cite[Lemma 3.9]{BL};
the proof given there works as well for the wheeled Kontsevich--LMO invariant. Applying this result, we get 
\begin{equation} \label{eq:uno}
Z^\wheel(K_1 \sqcup H \sqcup \epsilon_n K_2)
= \omega^2\, \Omega_{k_1}^{-1}\, \Omega_{k_2}^{-1}\, 
Z^\wheel(K_1)\, Z^\wheel(\widetilde{H})\vert_{h'\to k_1} \vert_{h'' \to k_2} \, Z^\wheel( \epsilon_n K_2)
\end{equation}
where $Z^\wheel( K_1 \sqcup H \sqcup \epsilon_n K_2) := Z^\wheel(M_1\sharp M_2, K_1 \sqcup H \sqcup \epsilon_n K_2)$. 
Besides, according to \cite[Prop.~4.3]{BL} which computes the value of $Z^\wheel$ on any Hopf chain, we have
\begin{eqnarray} \label{eq:dos}
Z^\wheel(\widetilde{H}) &=& \omega^{-1} \exp\Big(-\frac{\theta \sum_{i=1}^n a_i}{48}\Big) \prod_{i=1}^n \Omega_{h_i}^{-1} \cdot \\
\notag &&\cdot \exp\Big(\caps{h'}{h_{1}} + \sum_{i=1}^{n-1} \caps{h_i}{h_{i+1}} + \caps{h_n}{h''}
+\frac{1}{2} \sum_{i=1}^n a_i \caps{h_i}{h_{i}}  \Big).
\end{eqnarray}
Next, the way  the LMO invariant of a $\Q$HS can be computed from the wheeled Kontsevich integral 
of a surgery presentation in $S^3$ is described in  \cite[Eq.~(22)]{BL};
the proof given there works as well for a surgery presentation in any $\Q$HS other than $S^3$ 
--- see \cite[Remark 1.7]{BL} in this connection. Applying this result to $M=(M_1\sharp M_2)_{K_1\sqcup H\sqcup\epsilon_n K_2}$, 
we get the diagrammatic Gaussian integral
$$
Z(M)= \exp\Big(\frac{\theta \zeta(\Lambda) }{16}\Big) \int Z^\wheel(K_1 \sqcup H \sqcup \epsilon_n K_2)\, 
\Omega_{k_1}\, \Omega_{k_2}\, \prod_{i=1}^n \Omega_{h_i} \diff k_1 \diff h_1 \cdots \diff h_n \diff k_2,
$$ 
{where} $\zeta (\Lambda)$ denotes the signature of the linking matrix $\Lambda$ of $K_1 \sqcup H \sqcup \epsilon_n K_2$ in $M_1 \sharp M_2$.
Note that $\Lambda$ is the tridiagonal matrix
$$
\Lambda = \begin{pmatrix} 0  & 1  & 0 & \cdots & 0  &0 \\ 1 & a_1 & 1  & \cdots &0 & 0  \\
0 & 1 & a _2 &  \cdots & 0 &0 \\ \vdots & \vdots & \vdots & \ddots & \vdots & \vdots \\ 
 0 & 0 & 0  & \cdots & a_n & 1 \\ 0 & 0 & 0  & \cdots & 1 & 0 \end{pmatrix}.
$$
Combining this with \eqref{eq:uno} and \eqref{eq:dos}, we deduce that
\begin{eqnarray*}
Z(M)&=& \omega\,  \exp\Big(\frac{\theta (3\zeta(\Lambda) - \operatorname{tr}(\Lambda)) }{48}\Big)\,  \int  
Z^\wheel(K_1)\,   Z^\wheel(\epsilon_n K_2) \cdot \\
&&   \cdot \exp\Big(\caps{k_1}{h_{1}} + \sum_{i=1}^{n-1} \caps{h_i}{h_{i+1}} + \caps{h_n}{k_2}
+\frac{1}{2} \sum_{i=1}^n a_i \caps{h_i}{h_{i}}  \Big)   \diff k_1 \diff h_1 \cdots \diff h_n \diff k_2.
\end{eqnarray*}
Since $K_1$ is endowed with its preferred parallel, $Z^\wheel( K_1)$ has no strut;
similarly $Z^\wheel( \epsilon_n K_2)$ has no strut. Thus, the ``strut'' part of the above {integrand} is fully given by 
the ``$\exp$'' part that is apparent there. To compute  this Gaussian integral, we need to invert the tridiagonal matrix~$\Lambda$
but, knowing that $Z^\wheel(K_1)\,   Z^\wheel(\epsilon_n K_2)$ has no univalent vertex colored by $h_1,\dots,h_n$,
we only need to determine the four corners of $\Lambda^{-1}$. 
{To} this purpose, we compute the $2\times 2$ matrix associated to the tridiagonal matrix $\Lambda$
(see Appendix \ref{app:tri}):
\begin{eqnarray*}
&& \begin{pmatrix} 0 & -1 \\ 1 & 0 \end{pmatrix} \cdot \begin{pmatrix} 0 & -1 \\ 1 & 0 \end{pmatrix}
  \begin{pmatrix} a_1 & -1 \\ 1 & 0 \end{pmatrix} \cdots
  \begin{pmatrix} a_n & -1 \\ 1 & 0 \end{pmatrix} \begin{pmatrix} 0 & -1 \\ 1 & 0 \end{pmatrix}\\
&=& \begin{pmatrix} 0 & -1 \\ 1 & 0 \end{pmatrix} \cdot \begin{pmatrix} 0 & -1 \\ 1 & 0 \end{pmatrix} \cdot 
  \begin{pmatrix} 1 & 0 \\ 0 & -1 \end{pmatrix} \cdot 
  \begin{pmatrix} a_1 & 1 \\ -1 & 0 \end{pmatrix} \cdots
  \begin{pmatrix} a_n & 1 \\ -1 & 0 \end{pmatrix} \begin{pmatrix} 0 & 1 \\ -1 & 0 \end{pmatrix} 
  \cdot  \begin{pmatrix} 1 & 0 \\ 0 & -1 \end{pmatrix} \\
& \stackrel{\eqref{eq:decomposition_f}}{=} &   
\begin{pmatrix} 0 & -1 \\ 1 & 0 \end{pmatrix} \cdot \begin{pmatrix} 0 & -1 \\ 1 & 0 \end{pmatrix} \cdot 
  \begin{pmatrix} 1 & 0 \\ 0 & -1 \end{pmatrix} \cdot 
  \begin{pmatrix} p & r \\ q & s \end{pmatrix}^T \cdot \begin{pmatrix} 1 & 0 \\ 0 & -1 \end{pmatrix}
   \ = \
  \begin{pmatrix} -p & q \\ r & -s \end{pmatrix} 
  \end{eqnarray*} 
and we deduce from Proposition \ref{prop:BL} that
$$
\Lambda^{-1} = \begin{pmatrix}  p/r & ? & \cdots & ? & (-1)^{n+1}/r \\ ? & ? & \cdots & ? & ? \\
\vdots & \vdots & \ddots & \vdots  & \vdots \\
? & ? & \cdots & ? & ? \\ (-1)^{n+1}/r & ? & \cdots & ? & s/r
\end{pmatrix}.
$$
Therefore, performing Gaussian integration, we get
\begin{eqnarray*}
Z(M)&=& \omega\, \exp\Big(\frac{\theta (3\zeta(\Lambda) - \operatorname{tr}(\Lambda)) }{48}\Big) \cdot\\ 
&& \cdot \left\langle \exp\left( -\frac{p}{2r} \cups{k_1}{k_1} + \frac{(-1)^{n}}{r} \cups{k_1}{k_2} -\frac{s}{2r} \cups{k_2}{k_2} \right)
, Z^\wheel(K_1)\,   Z^\wheel( \epsilon_n K_2) \right\rangle_{}
\end{eqnarray*}
and, since $Z^\wheel( - K_2)$ is obtained from $Z^\wheel( K_2)$ by the change $k_2 \to -k_2$, this is equivalent to
\begin{eqnarray*}
Z(M)&=& \omega\, \exp\Big(\frac{\theta (3\zeta(\Lambda) - \operatorname{tr}(\Lambda)) }{48}\Big) \cdot\\ 
&& \cdot \left\langle \exp\left( -\frac{p}{2r} \cups{k_1}{k_1} - \frac{1}{r} \cups{k_1}{k_2} -\frac{s}{2r} \cups{k_2}{k_2} \right)
, Z^\wheel(K_1)\,   Z^\wheel( K_2) \right\rangle_{} \\
\end{eqnarray*}
whatever the parity of $n$ is. Using now Theorem \ref{th:KM},
we obtain
 \begin{eqnarray*}
 Z(M)&=& \omega\, \exp\left(\frac{\theta}{48}\left(-S\left(\frac pr\right)+\frac{p+s}r\right)\right)  \cdot\\
 && \cdot \left\langle\exp\left(-\frac{p}{2r} \cups{k_1}{k_1} - \frac{1}{r} \cups{k_1}{k_2} -\frac{s}{2r} \cups{k_2}{k_2}\right),
 Z^\wheel(K_1)\, Z^\wheel(K_2)\right\rangle.
 \end{eqnarray*}
Using the notation \eqref{eq:DD}, this is equivalent to 
$$
 Z(M) = \omega\, \exp\left(\frac{\theta}{48}\left(-S\left(\frac pr\right)+\frac{p+s}r\right)\right) 
 \left\langle\exp\left(-\frac 1r\cups{k_1}{k_2}\right),
 \partial_{D_1}\big(Z^\wheel(K_1)\big)\, \partial_{D_2}\big(Z^\wheel(K_2)\big)\right\rangle
$$
 and the conclusion easily follows.
\end{proof}

\begin{example}
When the splice of $(M_1,K_1)$ and $(M_2,K_2)$ is standard, 
Theorem~\ref{thFormuleSplicing} simplifies as follows.
Then $\mu(K_1)$ and $\rho(K_1)$ are glued along $\rho(K_2)$ and $\mu(K_2)$, respectively: thus $p=s=0$ and $q=-r=1$.
Hence, by identifying the symbols $k_1$ and $k_2$, we get
\begin{equation} \label{eq:standard}
 Z(M)= \omega\,  \big\langle  Z^\wheel(K_1)\,  ,\, Z^\wheel(K_2)  \big\rangle.
\end{equation}
\end{example}

Bar-Natan and Lawrence developed their techniques (which we have intensively used in the proof of Theorem~\ref{thFormuleSplicing})
to prove a \emph{rational} surgery formula for the LMO invariant \cite[Eq.~(23)]{BL}.
We now explain how their  formula --- in the case of a  rational surgery along a single knot --- 
is recovered from Theorem~\ref{thFormuleSplicing}.

\begin{corollary}[Bar-Natan \& Lawrence] \label{cor:BL}
Let $L\subset N$ be a null-homologous knot in a~$\Q$HS, with framing given by the preferred parallel.
Let $r,s$ be  non-zero coprime integers, 
and let $M$ be  obtained from $N$ by an $(r/s)$--surgery along $L$. Then
$$
Z(M)=  \exp\left({\frac{\theta}{48}\left({-S\left(\frac sr\right)+\frac{s}r}\right)}\right) \, 
 \left\langle \Omega_{l/r}\,  ,\, \partial_{D}\big(Z^\wheel(L)\big) \right\rangle
 \quad \hbox{with }
 D:= \exp\Big(-\frac{s}{2r}\cups{l}{l}\Big)
$$
and the latter identity is equivalent to \cite[Eq.~(23)]{BL}.
\end{corollary}

\begin{proof}
We consider $(M_1,K_1):=(S^3,\operatorname{unknot})$,
$(M_2,K_2):=(N,L)$ and the integers $p,q$ are chosen so that $ps-qr=1$. 
Then, the corresponding splice is the $3$--manifold $M$. Recall from~\eqref{eq:wheeled_Omega} that $Z^\wheel(K_1) = \omega^{-1} \Omega_{k_1}$.
Hence, in this particular situation, Theorem~\ref{thFormuleSplicing} gives 
 $$
 Z(M)=  \exp\left(\frac{\theta}{48}\left({-S\left(\frac pr\right)+\frac{p+s}r}\right)\right) \, 
 \left\langle \partial_{D_1}\big( \Omega_{k_1}\big)\big\vert_{k_1\to {-k_{2}}/{r}}\,  ,\, \partial_{D_2}\big(Z^\wheel(K_2)\big) \right\rangle
 $$
 where 
$$
 D_1 := \exp\Bigg(-\frac{p}{2r}\cups{k_1}{k_1}\Bigg) \quad \hbox{and} \quad 
 D_2 := \exp\Bigg(-\frac{s}{2r}\cups{k_2}{k_2}\Bigg).
$$
Besides, recall from \cite[Cor. 3.4]{BL} that, for any $\alpha \in \Q$, we have
\begin{equation} \label{eq:d_Omega}
\partial_E(\Omega) = \exp\Big(\frac{\alpha \theta}{48}\Big)\, \Omega \ \in \A(*)
\quad \hbox{where } E:= \exp\left(\frac{\alpha}{2}\cups{ }{ }\right).
\end{equation}
In particular, $\partial_{D_1}\Omega_{k_1}=\exp\left(-\frac pr\frac{\theta}{48}\right)\Omega_{k_1}$.
Hence, using that $\Omega$ only shows diagrams with an even number of univalent vertices
(see \eqref{eq:Omega}) and that $S(p/r)=S(s/r)$ (since $p\equiv s^{-1} \mod r$), we get
\begin{equation} \label{eq:B1}
 Z(M)=  \exp\left(\frac{\theta}{48}\left({-S\left(\frac sr\right)+\frac{s}r}\right)\right) \, 
 \left\langle \Omega_{l/r}\,  ,\, \partial_{D}\big(Z^\wheel(L)\big) \right\rangle.
\end{equation}
 
We now check the equivalence between this formula and \cite[Eq.~(23)]{BL}, which states (in the case of $N=S^3$) that
 \begin{equation} \label{eq:B2}
 Z(M) = \exp\left(\frac{\theta}{48} \Big(3 \operatorname{sign}\Big(\frac r s\Big) + S\Big(\frac r s\Big)  - \frac rs\Big) \right)
\left\langle \exp\Big(\frac{-s}{2r} \cups{l}{l}\Big), Z^\wheel(L)\,  \Omega_{l/s}\right\rangle.
\end{equation}
Note that
 $$
 \left\langle \exp\left(\frac{-s}{2r} \cups{l}{l}\right), Z^\wheel(L)\,  \Omega_{l/s}\right\rangle
 = \left\langle \partial_{D}( \Omega_{l/s})\big\vert_{l \to -sl/r} \, ,\, \partial_{D}\big(Z^\wheel(L)\big)  \right\rangle;
 $$
 besides, using \eqref{eq:d_Omega} one more time, we obtain
 $$
\partial_{D}(\Omega_{l/s})= \exp\left(-\frac{\theta}{48}\frac{1}{rs}\right) \Omega_{l/s};
 $$
 hence we deduce that 
  $$
 \left\langle \exp\left(\frac{-s}{2r} \cups{l}{l}\right), Z^\wheel(L)\,  \Omega_{l/s}\right\rangle
 = \exp\left(-\frac{\theta}{48}\frac{1}{rs}\right)
 \left\langle \Omega_{l/r} \, ,\, \partial_{D}(Z^\wheel(L))  \right\rangle.
 $$
 Then the equivalence between \eqref{eq:B1} and \eqref{eq:B2} 
 follows from the reciprocity law \eqref{eq:reciprocity} for Dedekind symbols.
\end{proof}

\begin{example}
Think of the lens space $L(r,s)$ as $S^3$ surgered along the $(r/s)$--framed unknot.
Then, a direct application of Corollary \ref{cor:BL} using \eqref{eq:wheeled_Omega} and \eqref{eq:d_Omega}, again, gives 
$$
Z\big(L(r,s)\big) = \omega^{-1}  
\exp\left(-\frac{\theta}{48}{S\left(\frac sr\right)}\right) \left\langle \Omega_{l/r} , \Omega_l \right\rangle.
$$
This formula is easily seen to be equivalent to \cite[Prop. 5.1]{BL}.
\end{example}

\section{Low degree formulas} \label{sec:low_degree}

In this section, we apply Theorem~\ref{thFormuleSplicing} to get splicing formulas for the lowest degree terms of the LMO invariant.
In the sequel, the \emph{degree} of a Jacobi diagram  refers to its \emph{internal} degree, 
{{\em ie}} its number of trivalent vertices, and the  symbol ``$\equiv_n$'' stands for an equality ``up to terms of degree at least $n$''.

According to \cite[Prop$.$ 5.3]{LMO}, the LMO invariant of any $\Q$HS $M$
can be written as follows up to degree $5$: 
\begin{equation} \label{eq:LowDeg}
Z(M)\equiv_6 \emptyset\, + \frac 14\lambda_{\operatorname{W}}(M)\teta+\lambda_2(M)\tetatwo+\frac 1{32}\lambda_{\operatorname{W}}(M)^2\teta\teta.
\end{equation}
Here $\lambda_{\operatorname{W}}(M)\in \Q$ is the Casson--Walker invariant of $M$ as normalized in \cite{Walker},
and $\lambda_2(M)\in\Q $ is a finite-type invariant of degree $4$ 
(which, up to a multiplicative constant, is the second term of the Ohtsuki series of $M$ \cite{Ohtsuki2}).

In \cite[Cor$.$ 1.2]{Fujita}, Fujita gives a formula for the Casson--Walker invariant of a $\Q$HS obtained as the splice of two knots in integral homology 3--spheres. The next proposition generalizes his result to the splice of null-homologous knots in $\Q$HS. Given a null-homologous knot $K$ in a $\Q$HS $M$, denote by 
$$
\Delta_K(t)\in\Qt
$$
its Alexander polynomial normalized so that $\Delta_K(t^{-1})=\Delta_K(t)$ and $\Delta_K(1)=1$ {(see \cite[Theorem 1.5]{Mou})}.

\begin{proposition}[Fujita] \label{PropFujita}
 Let $K_1 \subset M_1$ and $K_2 \subset M_2$ be null-homologous knots in $\Q$HS and let $M$ be the splice of $(M_1,K_1)$ and $(M_2,K_2)$ described by \eqref{eq:M_bis} in terms of four integers $p,q,r,s$. If $M$ is also a $\Q$HS, then  
 $$
 \lambda_{\operatorname{W}}(M)=\lambda_{\operatorname{W}}(M_1)+\lambda_{\operatorname{W}}(M_2) -\frac 1{12}S\left(\frac pr\right)
 +\frac{p}r\Delta''_{K_1}(1)+\frac{s}r\Delta''_{K_2}(1).
 $$
\end{proposition}

\begin{example}
For a standard splicing, we have 
$\lambda_{\operatorname{W}}(M)=\lambda_{\operatorname{W}}(M_1)+\lambda_{\operatorname{W}}(M_2)$.
This  additivity formula  for $\lambda_{\operatorname{W}}$ was obtained in \cite{BN,FN}.
\end{example}

\begin{remark}
In the case $(M_1,K_1):=(S^3,\operatorname{unknot})$, Proposition \ref{PropFujita} specializes 
to Walker's rational surgery formula for $\lambda_{\operatorname{W}}$ --- see \cite[Prop$.$ (6.2)]{Walker}.
\end{remark}

To express the splicing formula for $\lambda_2$, we need the following invariant of a null-homologous knot $K$ in a $\Q$HS $M$:
let $v(K)$ be the coefficient of $\tetalegs$ in $\chi^{-1}Z(K)$ where $Z(K)=Z(M,K)$ is the Kontsevich--LMO invariant.
When ${M=S^3}$, we have {$v(K)=-\frac{1}{24}j_3(K)$}
where $j_3(K)$ is the coefficient of $h^3$ in the Jones polynomial $V_K(t)$ evaluated at $t:=e^h$ (see \cite[Lemma~2.1]{Ito}).

\begin{proposition} \label{PropDegFour}
 Let $K_1 \subset M_1$ and $K_2 \subset M_2$ be null-homologous knots in $\Q$HS and let $M$ be the splice of $(M_1,K_1)$ and $(M_2,K_2)$ described by \eqref{eq:M_bis} in terms of four integers $p,q,r,s$. If $M$ is also a $\Q$HS, then 
 \begin{align*}
  \lambda_2(M)=\ & \lambda_2(M_1) +\lambda_2(M_2) + \frac{1}{1152}\Big(\frac{1}{r^2}-1\Big)\\
  &+\frac 1{96}\left(1-\frac 1{r^2}\right)\big(\Delta_{K_1}''(1)+\Delta_{K_2}''(1)\big)
  +\frac 9{16}\frac{p^2}{r^2}\Delta_{K_1}''(1)+\frac 9{16}\frac{s^2}{r^2}\Delta_{K_2}''(1) \\ 
  & +\frac 7{32}\frac{p^2}{r^2}\big(\Delta_{K_1}''(1)\big)^2+\frac 7{32}\frac{s^2}{r^2}\big(\Delta_{K_2}''(1)\big)^2
  +\frac 1{8r^2}\Delta_{K_1}''(1)\Delta_{K_2}''(1) \\
  & -\frac 5{96}\frac{p^2}{r^2} \Delta_{K_1}^{(4)}(1)-\frac 5{96}\frac{s^2}{r^2} \Delta_{K_2}^{(4)}(1)-\frac pr v(K_1)-\frac sr v(K_2).
 \end{align*}
\end{proposition}

\begin{example}
For a standard splicing, we have 
$\lambda_2(M)=\lambda_2(M_1)+\lambda_{2}(M_2)+ \frac 1{8}\Delta_{K_1}''\!(1)\,\Delta_{K_2}''\!(1)$.
We note that the same non-additivity formula is satisfied for the ``SU(3) Casson invariant'' 
in the case when $K_1$ and $K_2$ are torus knots in $S^3$ \cite{BH},
but the relation between this gauge-theoretical invariant and finite-type invariants does not seem to be known yet.
\end{example}

\begin{remark}
In the case $(M_1,K_1):=(S^3,\operatorname{unknot})$, Proposition \ref{PropDegFour} specializes 
to Ito's rational surgery formula for $\lambda_2$ --- see \cite[Theorem 1.2]{Ito}.
\end{remark}

To deduce from Theorem~\ref{thFormuleSplicing} the above splicing formulas for $\lambda_{\operatorname{W}}$ and $\lambda_2$,
we need first to determine the low-degree terms of the wheeled Kontsevich--LMO invariant $Z^\wheel$. The following result of Kricker (consequence of \cite[Theorem 1.0.8]{Kri}) gives the ``one-loop'' part of $Z^\wheel$. 
Kricker indeed works with a knot in {an integral homology $3$--sphere,} 
but the whole article adapts in the case of a null-homologous knot in a $\Q$HS, 
with the above normalization of the Alexander polynomial.

\begin{theorem}[Kricker] \label{thKricker}
 Let $K$ be a null-homologous knot in a $\Q$HS, with framing given by the preferred parallel. Then the Kontsevich--LMO invariant of $K$ writes $\chi^{-1}Z(K)=\operatorname{Wh}(K)\,\exp(R)$ where $R$ is a $\Q$--linear combination of connected Jacobi diagrams 
 with at least two loops and
 $$
 \operatorname{Wh}(K) := \Omega\exp\left( -\frac 12\log\left(\Delta_K(e^h)\right) \Big\vert_{h^{2m}\to\omega_{2m}}\right).
 $$
\end{theorem}

In the computations to come, it will be more convenient to use the Conway normalization 
$$
\nabla_K(z)=1+\sum_{m>0}a_{2m}(K)z^{2m}
$$ of the Alexander polynomial of a knot $K$, given by {$\nabla_K(t-t^{-1})=\Delta_K(t^2)$.
In particular, derivating four times this equality gives}
\begin{equation} \label{valuesDelta}
 \Delta_K''(1)=2a_2(K)\quad\textrm{ and }\quad\Delta_K^{(4)}(1)=24(a_2(K)+a_4(K)). 
\end{equation}
 
\begin{lemma}  \label{lem:calculs}
Let $K$ be a null-homologous knot in a $\Q$HS $M$, and let $a_{2i}:=a_{2i}(K)$ with $i\geq 1$ be the coefficients of the Alexander--Conway polynomial of $K$. Then 
 \begin{align*}
  \chi^{-1}Z(K)\equiv_6\ & \emptyset+ \frac{\lambda_{\operatorname{W}}(M)}{4}\teta+\lambda_2(M)\tetatwo+\frac{\lambda_{\operatorname{W}}(M)^2}{32}\teta\teta+\left(b_2-\frac 12a_2\right)\tetacup kk \\
  & +\frac{\lambda_{\operatorname{W}}(M)}{4}\left(b_2-\frac 12a_2\right)\teta\tetacup kk+\frac 12\left(b_2-\frac 12a_2\right)^2\tetacup kk\tetacup kk\\
  & +v(K)\tetaleg k+\left(b_4-\frac 1{24}a_2+\frac 14a_2^2-\frac 12a_4\right)\omegafour k
 \end{align*}
where  $b_2,b_4$ are the first  modified Bernoulli numbers \eqref{valuesb}.
\end{lemma}
\begin{proof}
 It is well-known (and easily verified from the AS and IHX relations) that the following diagrams constitute a basis of the ``strut-less'' \emph{and} ``connected'' part of $\A(*_k)$ up to degree~$5$:
 $$
 \emptyset, \quad \teta\ ,\ \tetacup{k}{k}, \quad  \ \tetatwo\ , \ \tetaleg k\ , \omegafour{k}
 $$
 Hence, by the group-like property of $\chi^{-1}Z(K)$ and the fact that its purely-trivalent part reduces to $Z(M)$, it suffices to show that its connected one-loop part equals
 $$
 \left(b_2-\frac 12a_2\right)\tetacup kk +\left(b_4-\frac 1{24}a_2+\frac 14a_2^2-\frac 12a_4\right)\omegafour k
 $$
 up to degree $5$.
 We have
 $$
 \Delta_K\big(e^h\big)=\nabla_K\big(e^{\frac h2}-e^{-\frac h2}\big)
 = 1+a_2h^2+\left(\frac{a_2}{12}+a_4\right)h^4+ (\deg \geq 6)
 $$
 thus
 $$
 \log\left(\Delta_K(e^h)\right)=a_2h^2+\left(\frac{a_2}{12}-\frac{a_2^2}2+a_4\right)h^4+ (\deg \geq 6)
 $$
 and we deduce that 
 $$
 \log \operatorname{Wh}(K)\equiv_6\left(b_2-\frac 12a_2\right)\omega_2
 +\left(b_4-\frac 1{24}a_2+\frac 14a_2^2-\frac 12a_4\right)\omega_4.
 $$
 We conclude thanks to Theorem~\ref{thKricker}.
\end{proof}

\begin{lemma} 
Let $K$ be a null-homologous knot in a $\Q$HS $M$, and let $a_{2i}:=a_{2i}(K)$ with $i\geq 1$ be the coefficients of the Alexander--Conway polynomial of $K$. Then 
 \begin{align*}
  Z^\wheel(K)\equiv_6\ & \emptyset+ \frac{\lambda_{\operatorname{W}}(M)}{4}\teta+\left(a_2b_2-2b_2^2+\lambda_2(M)\right)\tetatwo
  +\frac{\lambda_{\operatorname{W}}(M)^2}{32}\teta\teta \\
  & +\frac{\lambda_{\operatorname{W}}(M)}{4}\left(b_2-\frac 12a_2\right)\teta\tetacup kk+\frac 12\left(b_2-\frac 12a_2\right)^2\tetacup kk\tetacup kk\\
  & +\left(b_2-\frac 12a_2\right)\tetacup kk+v(K)\tetaleg k+\left(b_4-\frac 1{24}a_2+\frac 14a_2^2-\frac 12a_4\right)\omegafour k.
 \end{align*}
\end{lemma}

\begin{proof}
We have $\Omega^{-1} \equiv_3 \emptyset - b_2 \omega_2 $.
Since $\chi^{-1} Z(K)$ is strut-less, we have
$$
Z^\wheel(K) = \partial_{\Omega^{-1}}\big(\chi^{-1} Z(K)\big) \equiv_6
\chi^{-1} Z(K) -2b_2 \big(\hbox{coef. of $\omega_2$ in } \chi^{-1} Z(K)\big) \tetatwo
$$
and we conclude with Lemma \ref{lem:calculs}.
\end{proof}

We have now gathered all we need to prove the splicing formulas for $\lambda_{\operatorname{W}}$ and $\lambda_2$.

\begin{proof}[Proof of Proposition~\ref{PropFujita}]
 In the equality of Theorem~\ref{thFormuleSplicing}, the coefficient of $\theta$ in the left hand side 
 is $\lambda_{\operatorname{W}}(M)/4$, while the coefficient of $\theta$ in the right hand side is 
 \begin{eqnarray*} 
 && \frac 1{48}\left(\frac{s+p}r-S\left(\frac pr\right)\right)+\frac 14\lambda_{\operatorname{W}}(M_1)+\frac 14\lambda_{\operatorname{W}}(M_2)\\
 &&-\frac pr\left(b_2-\frac 12 a_2(K_1)\right) -\frac sr\left(b_2-\frac 12 a_2(K_2)\right).
 \end{eqnarray*}
 We conclude by substituting the values \eqref{valuesb} and \eqref{valuesDelta}.
\end{proof}

\begin{proof}[Proof of Proposition~\ref{PropDegFour}]
 We compute the coefficient of $\tetatwo$ in the right hand side  of the equality in Theorem~\ref{thFormuleSplicing}.
 First of all, the contribution of $\omega$ is $2b_2^2$. 
 Next, $\partial_{D_i}\big(Z^\wheel(K_i)\big)$ contributes in two ways to this coefficient:
 there is an ``invidual'' contribution through its $\tetatwo$ part, 
 and there is a ``mutual'' contribution  through its $\tetacup{}{}$ part.
 On the one hand,  the  individual contribution of  $\partial_{D_i}\big(Z^\wheel(K_i)\big)$ is
 $$
 a_2b_2-2b_2^2+\lambda_2(M_i)-\frac trv(K_i)+\frac{t^2}{ r^2}\left(b_2-\frac 12 a_2\right)^2
 +\frac 52\frac{t^2}{r^2}\left(b_4-\frac 1{24}a_2+\frac 14 a_2^2-\frac 12 a_4\right)
 $$
 where $a_\ell := a_\rho(K_i)$, $t := p$ if $i=1$ and $t:=s$ if $i=2$
 (to compute the sub-contribution resulting from the $\tetafour$ term of $Z^\wheel(K_i)$, we use an IHX relation that gives $\tetatwo=2\dfour$).
 On the other hand, the ``mutual'' contribution of $\partial_{D_1}\big(Z^\wheel(K_1)\big)$ and $\partial_{D_2}\big(Z^\wheel(K_2)\big)$ is
 $$\frac 2{r^2}\left(b_2-\frac 12 a_2(K_1)\right)\left(b_2-\frac 12 a_2(K_2)\right).$$
 Summing all these contributions and substituting the values \eqref{valuesb} and \eqref{valuesDelta}, we obtain the desired formula for $\lambda_2(M)$.
\end{proof}

\section{Splicing and satellite operations}
\label{sec:links}

The following generalization of Theorem \ref{thFormuleSplicing} is proved exactly in the same way.

\begin{theorem} \label{th:links}
 Let $K_1 \subset M_1$ and $K_2 \subset M_2$ be null-homologous knots in $\Q$HS 
 and let $M$ be the splice of $(M_1,K_1)$ and $(M_2,K_2)$  as described by \eqref{eq:M_bis} in terms of four integers $p,q,r,s$, 
 assuming that $r\neq0$.
 For each $i\in \{1,2\}$, let $L_i \subset M_i$ be a framed link disjoint from $K_i$ and let 
 $$
 Z^{\wheel k_i}(M_i,K_i\sqcup L_i) := \partial_{\Omega_{k_i}}^{-1} \chi^{-1}_{k_i} Z(M_i,K_i \sqcup L_i)
 $$
 be the Kontsevich--LMO invariant of $(M_i,K_i \sqcup L_i)$ ``wheeled'' at $K_i$. Then
 \begin{eqnarray*}
 Z(M,L_1 \sqcup L_2)&=& \omega\, \exp\left(\frac{\theta}{48}\left({-S\left(\frac pr\right)+\frac{p+s}r}\right)\right) \cdot\\
 & &\cdot \left\langle \partial_{D_1}\big(Z^{\wheel k_1}(M_1,K_1 \sqcup L_1)\big)\big\vert_{k_1\to {-k_{2}}/{r}}\, 
  ,\, \partial_{D_2}\big(Z^{\wheel k_2}(M_2,K_2\sqcup L_2)\big) \right\rangle
 \end{eqnarray*}
 where
 $$
 D_1 := \exp\Bigg(-\frac{p}{2r}\cups{k_1}{k_1}\Bigg) \quad \hbox{and} \quad 
 D_2 := \exp\Bigg(-\frac{s}{2r}\cups{k_2}{k_2}\Bigg).
 $$
\end{theorem}

\begin{example}
For a standard splice, identifying the symbols $k_1$ and $k_2$, we get
\begin{equation} \label{eq:standardlink}
 Z(M,L_1 \sqcup L_2)= \omega\,  \big\langle  Z^{\wheel k_1}(M_1,K_1 \sqcup L_1)\,  ,\, Z^{\wheel k_2}(M_2,K_2\sqcup L_2)  \big\rangle.
\end{equation}
\end{example}

Let $P$ be a framed link in the solid torus $S^1 \times D^2$ and let $L$ be a framed knot in an oriented $3$--manifold $N$.
We  identify $S^1 \times D^2$ with the tubular neighborhood $\neigh(L)$   so that $S^1 \times \{0\}$ and $S^1 \times \{1\}$  
correspond to the knot $L$ and its parallel $\rho(L)$, respectively.
The \emph{$P$--satellite} of $L$ is the  image of $P$ in $N$ by this  identification $S^1 \times D^2 \cong \neigh(L)$; this framed link is denoted by
$$L_P \subset N.$$ 
It is well-known that $L_P$ can  be constructed by splicing:
identify $S^1 \times D^2$ with the exterior of the unknot~$U \subset S^3$ so that both $S^1 \times \{0\}$ and $S^1 \times \{1\}$ 
are oriented meridians of $U$; then the standard splice of  $(N,L)$ and $(S^3,U)$ is a copy of $N$
where $P \subset S^3 \setminus \neigh(U)$ corresponds to~$L_P$.
Hence the following is an immediate consequence of Theorem \ref{th:links} in the simplified version of~\eqref{eq:standardlink}.

\begin{corollary} \label{cor:satellite}
Let $L\subset N$ be a null-homologous knot in a $\Q$HS (with framing given by the preferred parallel),
and let $P\subset S^1 \times D^2$ be a framed link. Then, by identifying the symbols $l$ and $u$, we have
\begin{equation} \label{eq:satellite}
Z(N,L_P) =  \omega\,  \big\langle  Z^\wheel(N,L)\,  ,\, Z^{\wheel u}(S^3,U\sqcup P)  \big\rangle.
\end{equation}
\end{corollary}

We  now derive from the ``satellite formula'' \eqref{eq:satellite}
a generalization of a  formula of Suetsugu~\cite{Suetsugu}.
The latter involves the extension of the Kontsevich integral
to framed links in a solid torus \cite{Suetsugu} or, equivalently, in a thickened annulus  \cite{AMR}. 
For an $n$--component framed link $P \subset S^1 \times D^2$, this  invariant  
$$
Z^\annulus(P) \in \A^\annulus(\circlearrowleft^1 \cdots \circlearrowleft^n)
$$
is valued in the space 
generated by homotopy classes of immersions from chord diagrams on $\circlearrowleft^1 \cdots \circlearrowleft^n$ to the annulus, modulo the 4T relation.

Let $D$ be a Jacobi diagram on $\circlearrowleft$ (possibly with some univalent vertices colored by a finite set $X$)
and let $E$ be a chord diagram  on $\circlearrowleft^1 \cdots \circlearrowleft^n$ in the annulus. 
One can produce from $D$ and $E$ a linear combination of Jacobi diagrams 
$$
D\vert_{\circlearrowleft \to E}
$$
on $\circlearrowleft^1 \cdots \circlearrowleft^n$ 
(possibly with some univalent vertices colored by $X$) by  the following procedure:
(1) thicken  the $1$-manifold $\circlearrowleft$ in $D$ to an annulus, replacing every univalent vertex of $D$ by a box directed in the same direction as $\circlearrowleft$; (2) put $E$ inside this annulus; (3)  apply the ``box notation'' below:
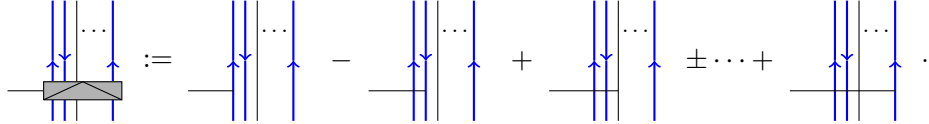
\begin{center}
\begin{tikzpicture} [scale=0.4]
\foreach \t in {0,6,12,18,26} {
\begin{scope} [xshift=\t cm,xscale=0.4]
 \foreach \x in {0,5} {\draw[thick,blue,->] (\x,0) -- (\x,2); \draw[thick,blue] (\x,2) -- (\x,4);}
 \draw[thick,blue,->] (1,4) -- (1,2); \draw[thick,blue] (1,2) -- (1,0);
 \draw (2,0) -- (2,4);
 \draw (3.6,3) node {\scalebox{0.8}{$\dots$}};
\end{scope}}
 \draw (-1.5,1) -- (0,1);
 \draw[xshift=6cm] (-1.5,1) -- (0,1);
 \draw[xshift=12cm] (-1.5,1) -- (0.4,1);
 \draw[xshift=18cm] (-1.5,1) -- (0.8,1);
 \draw[xshift=26cm] (-1.5,1) -- (2,1);
 \draw[fill=gray!60] (-0.3,0.7) -- (2.3,0.7) -- (2.3,1.3) -- (-0.3,1.3) -- (-0.3,0.7);
 \draw (-0.3,0.7) -- (1,1.3) -- (2.3,0.7);
 \draw (3.5,2) node {$:=$};
 \draw (9.6,2) node {$-$};
 \draw (15.6,2) node {$+$};
 \draw (22.5,2) node {$\pm\dots+$};
 \draw (29,2) node {.};
\end{tikzpicture}
\end{center}
Here, some pieces of chords of $E$ and  some arcs of the $1$-manifold $\circlearrowleft^1 \cdots \circlearrowleft^n$ 
(which are shown in bold blue)  go through the box,
and each of them contributes to one summand in the box notation;
an arc contributes with a $+$ or a $-$,
depending on the compatibility of its orientation with the direction
of the box; a piece of chord always contributes with a $+$, the orientation of the new trivalent vertex being determined by
the direction of the box. 
Of course, for step (2) to make sense, it is necessary to assume that (after a small homotopy) no univalent vertex of~$E$
(\emph{ie} no extremity of {any} chord) is contained in one of the boxes.
That the resulting operation 
$$
\A(\circlearrowleft, *_X) \times \A^\annulus(\circlearrowleft^1 \cdots \circlearrowleft^n) 
\longrightarrow \A(\circlearrowleft^1 \cdots \circlearrowleft^n,*_X), \quad (D,E) \longmapsto  D\vert_{\circlearrowleft \to E}
$$
is well-defined follows from the STU, {AS and IHX relations} in the target. 
Besides, it is obvious that this operation is compatible with additional ``link relations'' on $X$. 
This corresponds to the operation $\diamond$ in \cite{Suetsugu}. 

\begin{corollary}[Suetsugu]  \label{cor:satellite_bis}
Let $L\subset N$ be a null-homologous knot in a $\Q$HS (with framing given by the preferred parallel),
and let $P\subset S^1 \times D^2$ be a framed link. Then, we have
\begin{equation} \label{eq:satellite_bis}
Z(N,L_P) = \big(Z(N,L)\, \sharp\, \nu^{-1}\big)\big\vert_{\circlearrowleft \to Z^\annulus (P)}
\end{equation}
where $\sharp$ denotes the usual multiplication in $\A(\circlearrowleft) \cong \A(\uparrow)$ and $\nu=Z(\textrm{unknot})$.
\end{corollary}

\begin{proof}
 There is a tangle $Q$ in $ [-1,+1] \times D^2$   whose ``closure'' in $S^1 \times D^2$ is $P$. 
Thus, we obtain a decomposition of the following form for $U \sqcup P \subset S^3$ in the strict monoidal category of tangles:
\begin{equation} \label{eq:monoidal}
U \sqcup  P = \begin{array}{|c|} \hline \hbox{``caps''} \\ \hline \begin{array}{c|c} H & Q \end{array}\\
\hline \hbox{``cups''}\\ \hline \end{array} \,.
\end{equation}
Here ``caps'' (resp. ``cups'') denote the appropriate cabling {(with possible orientation changes)} of the elementary tangle 
\,\begin{tikzpicture} [scale=0.25]
 \draw[->] (1,0) arc (0:180:1);
\end{tikzpicture}
(resp.
\begin{tikzpicture} [scale=0.25]
 \draw[->] (-1,0) arc (-180:0:1);
\end{tikzpicture}\,), and $H$ is the ``open'' Hopf link  whose ``opened'' component has been cabled accordingly:
$$H:=\raisebox{-0.8cm}
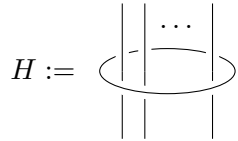
{
\begin{tikzpicture} [scale=0.3]
 \foreach \x in {-2,-1,2} {\draw (\x,-3) -- (\x,0);}
 \draw[white,line width=5pt] (0,0) circle (3 and 1);
 \draw (0,0) circle (3 and 1);
 \foreach \x in {-2,-1,2} {\draw[white,line width=5pt] (\x,0) -- (\x,3); \draw (\x,0) -- (\x,3);}
 \draw (0.5,2) node {$\dots$};
\end{tikzpicture}}$$ 
In the decomposition \eqref{eq:monoidal}, $U$ is the closed component of $H$.
By choosing a parenthesization of the top boundary points of $Q$, we can upgrade \eqref{eq:monoidal}
to a decomposition in the non-strict monoidal category of $q$--tangles. 
The value of $Z(H)$ is deduced from \cite[Theorem 4]{BL} by cabling the opened component of the ``open'' Hopf link. 
Then, using the definition of the invariant $Z^\annulus$
as it is done in \cite{AMR}, we obtain that
$$
\chi^{-1}_u Z(S^3,U \sqcup P) = 
\Omega_u \cdot \exp_\sharp\big(\,{}^u\!\!\!\figstrut \big) \big\vert_{\circlearrowleft \to Z^\annulus(P)}.
$$
Then, we deduce from formula \eqref{eq:satellite} that 
\begin{eqnarray*}
Z(N,L_P) &=&  \omega\,  \Big\langle  Z^\wheel(N,L) \cdot  \Omega_l^{-1} \,  ,\, 
\Omega_l \cdot \exp_\sharp\big(\,{}^l\!\!\!\figstrut \big) \big\vert_{\circlearrowleft \to Z^\annulus(P)}  \Big\rangle \\
&=&  \omega\,  \Big\langle  \partial_{\Omega_l}\big(Z^\wheel(N,L) \cdot  \Omega_l^{-1}\big) \,  ,\, 
 \exp_\sharp\big(\,{}^l\!\!\!\figstrut \big) \big\vert_{\circlearrowleft \to Z^\annulus(P)}  \Big\rangle 
\end{eqnarray*}
and, by  {``wheeling'' \cite[Theorem 3]{BL}}, we get
\begin{eqnarray*}
Z(N,L_P) &=& \omega\,  \Big\langle  \chi_l^{-1} \big(Z (N,L)\, \sharp \,  \chi_l \partial_{\Omega_{l}} (\Omega_l^{-1})\big) \,  ,\, 
 \exp_\sharp\big(\,{}^l\!\!\!\figstrut \big) \big\vert_{\circlearrowleft \to Z^\annulus(P)}  \Big\rangle \\
 &=& \omega\,  \Big\langle  \chi_l^{-1} \big(Z (N,L)\, \sharp\,  \big(\chi_l \partial_{\Omega_{l}} (\Omega_l)\big)^{-1}\big) \,  ,\, 
 \exp_\sharp\big(\,{}^l\!\!\!\figstrut \big) \big\vert_{\circlearrowleft \to Z^\annulus(P)}  \Big\rangle.
\end{eqnarray*}
By {\cite[Prop. 3.3]{BL}}, we have $\partial_{\Omega}(\Omega)= \omega \Omega$. Hence $\chi\partial_{\Omega}(\Omega)= \omega \nu$ and we conclude that
$$
Z(N,L_P) =  \Big\langle  \chi_l^{-1} \big(Z (N,L) \sharp \nu^{-1} \big)\,  ,\, 
 \exp_\sharp\big(\,{}^l\!\!\!\figstrut \big) \big\vert_{\circlearrowleft \to Z^\annulus(P)}  \Big\rangle
= \big(Z(N,L)\, \sharp\, \nu^{-1}\big)\big\vert_{\circlearrowleft \to Z^\annulus (P)}.
$$
\end{proof}

\begin{remark}
Formula {\eqref{eq:satellite_bis}} is stated in \cite[Theorem 2.5]{Suetsugu} for the case $N:=S^3$.
An indication of proof is  given, but the factor ``\,$\sharp \nu^{-1}$\,'' seems to be missing there.
\end{remark}

\begin{remark}
If $N:=S^3$ and $P$ is connected in Corollary \ref{cor:satellite_bis}, 
then we can view $L$  and~$P$ as one-component ``bottom tangles'' in  handlebodies of genus $0$ and $1$, respectively.
In such a case, formula {\eqref{eq:satellite_bis}} follows immediately from the functoriality
of the invariant of ``bottom tangles in handlebodies'' that is constructed in \cite{HM} as an extension of the Kontsevich integral. 
\end{remark}

\section{The case of knots that are not trivial in homology}
\label{sec:non-null}

Let  now $M_1$ and $M_2$ be $\Q$HS
and let $K_1 \subset M_1$ and $K_2 \subset M_2$ be framed knots. 
We identify the boundary of $X_i= M_i\setminus\Int(\neigh(K_i))$ with $\torus$ as in \eqref{eq:tori}
and we consider the splice 
\begin{equation} \label{eq:M_ter}
M:= X_1\, \mathop{\bigcup}_{f}\, X_2
\end{equation}
defined by the homeomorphism $f:\torus \to \torus$
that is encoded by four integers $p,q,r,s$ as in \eqref{eq:matrix_f}.
The self-linking number of $K_i$ is denoted by
\begin{equation} \label{eq:framing}
\frac{u_i}{v_i} := \lk\big(K_i,\rho(K_i)\big) \in \Q,
\end{equation}
where $u_i,v_i$ are coprime integers such that $v_i>0$.
Observe that, when $K_i$ is null-homologous and $\rho(K_i)$ is the preferred parallel, we have $u_i=0$ and $v_i=1$.

\begin{lemma} \label{lem:MV}
With the above notations, $M$ is a $\Q$HS if and only if we have
$$
q\,u_1u_2+r\,v_1v_2+s\,u_2v_1+p\,v_2u_1 \neq 0.
$$
\end{lemma}
\begin{proof}
Recall that the \emph{longitude} $\ell(K_i)$ of $K_i$ is the  oriented simple closed curve in $\partial  \neigh(K_i)$
that is homologous to  a positive multiple of $K_i$ in $\neigh(K_i)$ and  is rationally null-homologous in the exterior of $\neigh(K_i)$.
Thus we have $\lk(K_i, \ell(K_i))=0$ and we deduce from \eqref{eq:framing} that
$$
\ell(K_i) = -u_i\, \mu(K_i) +v_i\, \rho(K_i) \ \in H_1(\partial \neigh(K_i);\Z).
$$
A Mayer--Vietoris argument shows that $H_1(M;\Q)$ is generated by $\mu(K_2)$ and $\rho(K_2)$, subject to the relations
$$
-u_2\, \mu(K_2) +v_2\, \rho(K_2) =0 \quad \hbox{and} \quad (-pu_1 -rv_1)\, \mu(K_2) + (-qu_1-sv_1)\, \rho(K_2)=0
$$
which express the triviality in homology of $\ell(K_2)$ and $\ell(K_1)$, respectively. Hence $H_1(M_;\Q)=0$  if and only if 
$$
\begin{vmatrix}
-u_2 & v_2 \\
pu_1+rv_1 & qu_1+sv_1
\end{vmatrix} \neq 0
$$
\end{proof}

For any framed knot $L$ in a $\Q$HS $N$, with self-linking number $\varphi \in \Q$, we denote by 
$$
\underline{Z}(N,L) :=  Z(N,L) \sharp \exp_\sharp\Big( -\frac{\varphi}{2}\,\raisebox{-1ex}{ 
\begin{tikzpicture} [scale=0.3]
 \draw[->,thick, blue] (1,-1) -- (1,1);
 \draw (0,0) -- (1,0);
\end{tikzpicture}}\ \Big)
\ \in \A(\circlearrowleft_l) \cong \A(\uparrow^l) 
$$ 
the \emph{unframed} version of the Kontsevich--LMO invariant $Z(N,L)$ and we denote by
$$
\underline{Z}^\wheel(N,L) := \partial_{\Omega_l}^{-1} \chi^{-1}(\underline{Z}(N,L)) \ \in \A(*_l)
$$
the wheeled version of the latter. Note that $\underline{Z}^\wheel(N,L)$ does not show any strut and, by ``wheeling'' \cite[Theorem 3]{BL}, we have
\begin{equation} \label{eq:strut}
\underline{Z}^\wheel(N,L) = {Z}^\wheel(N,L) \cdot \exp\Big(-\frac{\varphi}{2} \caps{l}{l} + \frac{\varphi}{48} \teta \Big).
\end{equation}

\begin{theorem} \label{th:non-null}
For $i\in\{1,2\}$, let $K_i\subset M_i$ be a framed knot in a $\Q$HS, with self-linking  number $u_i/v_i\in \Q$, 
 and let $M$ be the splice of $(M_1,K_1)$ and $(M_2,K_2)$ as described by \eqref{eq:M_ter} in terms of four integers $p,q,r,s$. 
 If~$M$ is also a $\Q$HS,  then we have  
 $$
 Z(M)= \omega\, \exp\left(\frac{\theta \kappa}{48}  \right) \, 
 \left\langle \partial_{D_1}\big( \underline{Z}^\wheel(K_1) \big)\big\vert_{k_1\to -v_1v_2 k_2 /\lambda}\,  
 ,\, \partial_{D_2}\big( \underline{Z}^\wheel(K_2) \big) \right\rangle, 
 $$
where we have set $\lambda := q u_1u_2 + rv_1v_2 + s u_2 v_1  + pv_2u_1 $, $\tau_1 :=  qu_1 +sv_1$,  $\tau_2 :=  qu_2 +pv_2$,
\begin{eqnarray*} 
\kappa &:=& \left\{\begin{array}{ll}  S(s/q) - (s+p)/q  +3 \operatorname{sgn}(q\tau_1) +3 \operatorname{sgn}(\lambda\tau_1) -u_1/v_1-u_2/v_2 & \hbox{if } q \neq 0, \\
s \big(3 \operatorname{sgn}(\lambda) -r\big) -u_1/v_1-u_2/v_2  &  \hbox{if } q =  0, \end{array}\right.
\end{eqnarray*} 
and 
 \begin{equation}
 \label{eq:DD_bis}  
 D_1 := \exp\Bigg(-\frac{v_1 \tau_2}{2\lambda} \cups{k_1}{k_1} \Bigg), \qquad
D_2 := \exp\Bigg( -\frac{v_2 \tau_1}{2\lambda } \cups{k_2}{k_2} \Bigg).
  \end{equation}
\end{theorem}
\begin{proof}
The first part of the proof of Theorem \ref{thFormuleSplicing} works verbatim in the more general situation that we are now considering.
Thus we choose a matrix decomposition as in \eqref{eq:decomposition_f}, which leads to $n\geq 1$ integers $a_1,\dots, a_n$,
and we consider the Hopf chain ${H} := H(a_1,\dots,a_n) \subset M_1 \sharp M_2$ that ``clasps'' positively $K_1$
and $\epsilon_nK_2$. We obtain that
\begin{eqnarray*}
Z(M)&=& \omega\,  \exp\Big(\frac{\theta \big(3\zeta(\Lambda) - \sum_{i=1}^n a_i\big) }{48}\Big)\,  \int  
Z^\wheel(K_1)\,   Z^\wheel(\epsilon_n K_2) \cdot \\
&&   \cdot \exp\Big(\caps{k_1}{h_{1}} + \sum_{i=1}^{n-1} \caps{h_i}{h_{i+1}} + \caps{h_n}{k_2}
+\frac{1}{2} \sum_{i=1}^n a_i \caps{h_i}{h_{i}}  \Big)   \diff k_1 \diff h_1 \cdots \diff h_n \diff k_2.
\end{eqnarray*}
Here $\Lambda$ denotes the linking matrix of $K_1 \sqcup H \sqcup \epsilon_n K_2$ in  $M_1 \sharp M_2$, which is the tridiagonal matrix
$$
\Lambda := \begin{pmatrix} u_1/v_1  & 1  & 0 & \cdots & 0  &0 \\ 1 & a_1 & 1  & \cdots &0 & 0  \\
0 & 1 & a _2 &  \cdots & 0 &0 \\ \vdots & \vdots & \vdots & \ddots & \vdots & \vdots \\ 
 0 & 0 & 0  & \cdots & a_n & 1 \\ 0 & 0 & 0  & \cdots & 1 & u_2/v_2 \end{pmatrix}.
$$
Using \eqref{eq:strut}, we deduce that
\begin{eqnarray*}
Z(M)&=& \omega\,  \exp\Big(\frac{\theta \big(3\zeta(\Lambda) - \operatorname{tr}(\Lambda) \big) }{48}\Big)\,  \int  
\underline{Z}^\wheel(K_1)\,   \underline{Z}^\wheel(\epsilon_n K_2) \cdot \\
&&   \cdot \exp\Big(\caps{k_1}{h_{1}} + \sum_{i=1}^{n-1} \caps{h_i}{h_{i+1}} + \caps{h_n}{k_2}
+\frac{1}{2} \sum_{i=1}^n a_i \caps{h_i}{h_{i}} + \frac{u_1}{2 v_1} \caps{k_1}{k_1} + \frac{u_2}{2 v_2} \caps{k_2}{k_2} \Big) \cdot  \\
&& \cdot \diff k_1 \diff h_1 \cdots \diff h_n \diff k_2
\end{eqnarray*}
The $2\times 2$ matrix associated to the tridiagonal matrix $\Lambda$ is
\begin{eqnarray*}
&& \begin{pmatrix} 0 & -1 \\ 1 & 0 \end{pmatrix} \cdot \begin{pmatrix} u_1/v_1 & -1 \\ 1 & 0 \end{pmatrix}
  \begin{pmatrix} a_1 & -1 \\ 1 & 0 \end{pmatrix} \cdots
  \begin{pmatrix} a_n & -1 \\ 1 & 0 \end{pmatrix} \begin{pmatrix} u_2/v_2 & -1 \\ 1 & 0 \end{pmatrix}\\
&=& \begin{pmatrix} 0 & -1 \\ 1 & 0 \end{pmatrix}  \begin{pmatrix} u_1/v_1 & -1 \\ 1 & 0 \end{pmatrix}
  \begin{pmatrix} 1 & 0 \\ 0 & -1 \end{pmatrix}  \begin{pmatrix} a_1 & 1 \\ -1 & 0 \end{pmatrix} \cdots
  \begin{pmatrix} a_n & 1 \\ -1 & 0 \end{pmatrix} \begin{pmatrix} 1 & 0 \\ 0 & -1 \end{pmatrix}  
  \begin{pmatrix} u_2/v_2 & -1 \\ 1 & 0 \end{pmatrix} \\
  &=& \begin{pmatrix} 0 & -1 \\ 1 & 0 \end{pmatrix}  \begin{pmatrix} u_1/v_1 & 1 \\ 1 & 0 \end{pmatrix}
\Bigg( \begin{pmatrix} a_n & -1 \\ 1 & 0 \end{pmatrix} \cdots   \begin{pmatrix} a_1 & -1 \\ 1 & 0 \end{pmatrix} \Bigg)^T   
 \begin{pmatrix} u_2/v_2 & -1 \\ -1 & 0 \end{pmatrix} \\
 &\stackrel{\eqref{eq:decomposition_f}}{=}&   
 \begin{pmatrix} 0 & -1 \\ 1 & 0 \end{pmatrix}  \begin{pmatrix} u_1/v_1 & 1 \\ 1 & 0 \end{pmatrix}
 \begin{pmatrix} q  & s \\ -p  & -r \end{pmatrix}^T   
 \begin{pmatrix} u_2/v_2 & -1 \\ -1 & 0 \end{pmatrix} \\
 & = &  
 \begin{pmatrix} -(qu_2+pv_2)/{v_2} & q \\  (q u_1u_2 + s u_2 v_1  + pu_1v_2 + rv_1v_2) /(v_1v_2)& -(qu_1+sv_1) /v_1  \end{pmatrix}
\end{eqnarray*}
so that, by Proposition \ref{prop:BL}, we have 
$$
\Lambda^{-1} = \frac{1}{\lambda} \begin{pmatrix}  v_1 \tau_2 & ? & \cdots & ? &  (-1)^{n+1} v_1v_2 \\ ? & ? & \cdots & ? & ? \\
\vdots & \vdots & \ddots & \vdots  & \vdots \\
? & ? & \cdots & ? & ? \\ (-1)^{n+1} v_1v_2 & ? & \cdots & ? & v_2 \tau_1
\end{pmatrix}
$$
where $\lambda := q u_1u_2 + s u_2 v_1  + pu_1v_2 + rv_1v_2$, $\tau_1 :=  qu_1 +sv_1$
and  $\tau_2 :=  qu_2 +pv_2$.
Therefore, performing Gaussian integration and using that $\underline{Z}^\wheel( - K_2)$ 
is obtained from $\underline{Z}^\wheel( K_2)$ by the change $k_2 \to -k_2$, we get
\begin{eqnarray*}
Z(M)&=& \omega\, \exp\Big(\frac{\theta (3\zeta(\Lambda) - \operatorname{tr}(\Lambda)) }{48}\Big) \cdot\\ 
&& \cdot \left\langle \exp\left( -\frac{v_1 \tau_2}{2\lambda} \cups{k_1}{k_1} - \frac{v_1v_2}{\lambda} \cups{k_1}{k_2} 
-\frac{v_2\tau_1}{2\lambda} \cups{k_2}{k_2} \right)
, \underline{Z}^\wheel(K_1)\,  \underline{Z}^\wheel(K_2) \right\rangle_{}.
\end{eqnarray*}
Using the notation \eqref{eq:DD_bis}, this is equivalent to 
$$
 Z(M)= \omega\, \exp\Big(\frac{\theta (3\zeta(\Lambda) - \operatorname{tr}(\Lambda)) }{48}\Big)  \cdot 
 \left\langle \partial_{D_1}\big( \underline{Z}^\wheel(K_1) \big)\big\vert_{k_1\to -v_1v_2 k_2 /\lambda}\,  
 ,\, \partial_{D_2}\big( \underline{Z}^\wheel(K_2) \big) \right\rangle.
$$

It  remains to compute   $\kappa:= 3\zeta(\Lambda) - \operatorname{tr}(\Lambda) \in \Q$. 
By the above computation of the  $2 \times 2$ matrix  $A(u_1/v_1,a_1,\dots,a_n,u_2/v_2)$ 
associated to $\Lambda= \Lambda(u_1/v_1,a_1,\dots,a_n,u_2/v_2)$, 
we deduce from \eqref{eq:right} that 
\begin{eqnarray*}
\zeta(\Lambda)  &=& \zeta(\Lambda(u_1/v_1,a_1,\dots,a_n)) + \operatorname{sgn}(\lambda \tau_1).
\end{eqnarray*}
Besides, since the $2 \times 2$ matrix associated to  $\Lambda(u_1/v_1,a_1,\dots,a_n)$ is
\begin{eqnarray*}
&& A(u_1/v_1,a_1,\dots,a_n,u_2/v_2) \begin{pmatrix} 0 & 1 \\ -1 & u_2/v_2 \end{pmatrix} \\
&=&  \begin{pmatrix} -\tau_2/{v_2} & q \\  \lambda /(v_1v_2)& - \tau_1 /v_1  \end{pmatrix}
\begin{pmatrix} 0 & 1 \\ -1 & u_2/v_2 \end{pmatrix} 
\ = \ \begin{pmatrix} -q& -p \\  \tau_1/v_1 & (pu_1 + rv_1  )/v_1 \end{pmatrix},
\end{eqnarray*}
it follows from \eqref{eq:left} that
$$
\zeta(\Lambda(u_1/v_1,a_1,\dots,a_n)) = \zeta(\Lambda(a_1,\dots,a_n)) + \operatorname{sgn}(q\tau_1).
$$
Finally, since the $2 \times 2$ matrix associated to  $\Lambda(a_1,\dots,a_n)$ is
\begin{eqnarray*}
&& \begin{pmatrix} 0 & -1 \\ 1 & 0 \end{pmatrix} \cdot   \begin{pmatrix} a_1 & -1 \\ 1 & 0 \end{pmatrix} \cdots
  \begin{pmatrix} a_n & -1 \\ 1 & 0 \end{pmatrix} \\
&=& \begin{pmatrix} 0 & -1 \\ 1 & 0 \end{pmatrix} 
  \begin{pmatrix} 1 & 0 \\ 0 & -1 \end{pmatrix}  \begin{pmatrix} a_1 & 1 \\ -1 & 0 \end{pmatrix} \cdots
  \begin{pmatrix} a_n & 1 \\ -1 & 0 \end{pmatrix} \begin{pmatrix} 1 & 0 \\ 0 & -1 \end{pmatrix}   \\
 &\stackrel{\eqref{eq:decomposition_f}}{=}&   \begin{pmatrix} 0 & 1 \\ 1 & 0 \end{pmatrix} 
 \begin{pmatrix} q  & s \\ -p  & -r \end{pmatrix}^T   \begin{pmatrix} 1 & 0 \\ 0 & -1 \end{pmatrix} 
 \ = \  \begin{pmatrix} s  & r \\ q & p \end{pmatrix},
\end{eqnarray*}
we deduce from Theorem \ref{th:KM} that
$$
3 \zeta\big(\Lambda(a_1,\dots,a_n)\big)- \sum_{i=1}^n a_i =
\left\{\begin{array}{ll}  S(s/q) - (s+p)/q & \hbox{if } q \neq 0, \\
-r/s  &  \hbox{if } q =  0. \end{array}\right.
$$
We conclude that
\begin{eqnarray*} 
\kappa &=& \left\{\begin{array}{ll}  S(s/q) - (s+p)/q  +3 \operatorname{sgn}(q\tau_1) +3 \operatorname{sgn}(\lambda\tau_1) -u_1/v_1-u_2/v_2 & \hbox{if } q \neq 0, \\
-r/s   +3 \operatorname{sgn}(\lambda s) -u_1/v_1-u_2/v_2  &  \hbox{if } q =  0. \end{array}\right.
\end{eqnarray*}
\end{proof}

\begin{remark}
At the end of the proof of Theorem \ref{th:non-null}, we have applied \eqref{eq:right} before \eqref{eq:left} to compute $\kappa \in \Q$.
If we went the other way, we would have obtained an equivalent formula:
\begin{eqnarray*} 
\kappa &=& \left\{\begin{array}{ll}  S(s/q) - (s+p)/q  +3 \operatorname{sgn}(q\tau_2) +3 \operatorname{sgn}(\lambda\tau_2) -u_1/v_1-u_2/v_2 & \hbox{if } q \neq 0, \\
p \big(3 \operatorname{sgn}(\lambda) -r\big) -u_1/v_1-u_2/v_2  &  \hbox{if } q =  0. \end{array}\right.
\end{eqnarray*} 
\end{remark}

\begin{remark}
In the same way as Theorem \ref{thFormuleSplicing} extends to Theorem \ref{th:links},
there is a generalization of Theorem \ref{th:non-null} which includes additional links~$L_i$ in $M_i$ (disjoint from $K_i$). 
\end{remark}

\appendix

\section{Signatures of tridiagonal matrices} \label{app:tri}

A \emph{tridiagonal matrix} is a square matrix of the form
$$
\Lambda(c_1, \dots, c_\ell) :=
\begin{pmatrix}
c_1 & 1 & 0 & \ldots & \ldots &0 \\
1 & c_2 & 1 & \ddots & & \vdots \\
0 & 1& c_3 &1  & \ddots & \vdots \\
\vdots & \ddots & \ddots & \ddots & \ddots & 0 \\
\vdots & & \ddots & \ddots & \ddots & 1 \\
0 & \ldots & \ldots & 0 & 1 & c_\ell \\
\end{pmatrix}
$$
where $c_1,\dots,c_\ell\in \mathbb{R}$. The \emph{associated} $2\times 2$ matrix is
$$
\operatorname{A}(c_1,\dots,c_\ell) := 
\begin{pmatrix} 0 & -1 \\ 1 & 0 \end{pmatrix} \cdot \begin{pmatrix} c_1 & -1 \\ 1 & 0 \end{pmatrix}
\begin{pmatrix} c_2 & -1 \\ 1 & 0 \end{pmatrix} \cdots \begin{pmatrix} c_\ell & -1 \\ 1 & 0 \end{pmatrix}
\ \in \operatorname{SL}_2(\mathbb{R}).
$$
\begin{proposition}[Bar-Natan \& Lawrence] \label{prop:BL}
Let $c_1,\dots,c_\ell\in \mathbb{R}$ and write
$$
\begin{pmatrix} \alpha & \beta \\ \gamma & \delta \end{pmatrix} := \operatorname{A}(c_1,\dots,c_\ell).
$$
The matrix  $\Lambda(c_1, \dots, c_\ell)$ is invertible if and only if $\gamma\neq 0$.
In such a case, the four ``corners'' of the inverse of  $\Lambda(c_1, \dots, c_\ell)$ are given by
$$
\Lambda(c_1, \dots, c_\ell)^{-1} = \begin{pmatrix}  - \alpha/\gamma & ? & \cdots & ? & (-1)^{\ell+1}/\gamma \\ ? & ? & \cdots & ? & ? \\
\vdots & \vdots & \ddots & \vdots  & \vdots \\
? & ? & \cdots & ? & ? \\ (-1)^{\ell +1}/\gamma & ? & \cdots & ? & -\delta/\gamma
\end{pmatrix}.
$$
\end{proposition}
 
\begin{proof}
The  arguments given in \cite[Prop. 2.4]{BL} under the assumption  $c_1,\dots,c_\ell\in \mathbb{Z}$
work in the same way  for arbitrary $c_1,\dots,c_\ell\in \mathbb{R}$.
But it seems that the $(-1)^\ell$ in the statement of \cite[Prop. 2.4]{BL} should be replaced by $(-1)^{\ell+1}$.
\end{proof}

We explain how to compute inductively the signature $\zeta(\Lambda)$ of a tridiagonal matrix $\Lambda$.
\begin{proposition} \label{prop:signatures}
Let $c_1,\dots,c_\ell\in \mathbb{R}$ and write
$$
\begin{pmatrix} \alpha & \beta \\ \gamma & \delta \end{pmatrix} := \operatorname{A}(c_1,\dots,c_\ell).
$$
Then we have the following inductive formulas:
\begin{eqnarray}
\label{eq:right} \zeta\big(\Lambda(c_1,c_2, \dots,c_\ell) \big)&= & \zeta\big(\Lambda(c_1, \dots,c_{\ell-1}) \big) 
-\operatorname{sgn}(\gamma)\, \operatorname{sgn}(\delta),\\
\label{eq:left} \zeta\big(\Lambda(c_1,c_2, \dots,c_\ell) \big)&= & \zeta\big(\Lambda(c_2, \dots,c_\ell) \big) 
-\operatorname{sgn}(\gamma)\, \operatorname{sgn}(\alpha).
\end{eqnarray}
\end{proposition}

\begin{proof}
The proof of \eqref{eq:right} being  very similar to the proof of \eqref{eq:left}, we give the latter and omit the former.
A straightforward computation gives
\begin{equation} \label{eq:A}
A(c_2,\dots,c_\ell) = \begin{pmatrix} \gamma + \alpha c_1 &  \delta + \beta c_1 \\ -\alpha & -\beta \end{pmatrix}.
\end{equation}
Thus, by Proposition \ref{prop:BL}, the matrix $\underline{\Lambda} := \Lambda(c_2, \dots,c_\ell)$ is invertible if and only if $\alpha\neq 0$.

Assume first that $\alpha\neq 0$.  
We shall use the fact that ${\Lambda} := \Lambda(c_1, c_2,\dots,c_\ell)$ is obtained from~$\underline{\Lambda}$
by a kind of ``plumbing'' operation --- compare with the proof of \cite[Prop$.$ 2.5.3]{GR}.
Specifically, we have the following congruence where  $v:=(1,0,\dots,0)\in \mathbb{R}^{\ell-1}$:
$$
\underbrace{\begin{pmatrix} c_1  & v \\ v^T & \underline{\Lambda} \end{pmatrix}}_{\Lambda}
= \begin{pmatrix} 1 & v \underline{\Lambda}^{-1} \\ 0  & I \\ \end{pmatrix}  
\begin{pmatrix} c_1 - v  \underline{\Lambda}^{-1} v^T & 0 \\ 0 & \underline{\Lambda} \end{pmatrix}
\begin{pmatrix} 1 & 0 \\ \underline{\Lambda}^{-1} v^T & I \\ \end{pmatrix}.
$$
Therefore 
$$
\zeta(\Lambda) = \operatorname{sgn}(c_1 - v  \underline{\Lambda}^{-1} v^T) + \zeta(\underline{\Lambda} ).
$$
By Proposition \ref{prop:BL}, the upper left ``corner'' of $\underline{\Lambda}^{-1}$ can be read from \eqref{eq:A}.
Hence we get $v  \underline{\Lambda}^{-1} v^T = \gamma/\alpha +c_1 $ and we deduce that 
$\zeta(\Lambda) = - \operatorname{sgn}(\gamma/\alpha) + \zeta(\underline{\Lambda} )$.

Assume now that $\alpha=0$. Let $x=(x_2,\dots,x_\ell) \in \mathbb{R}^{\ell-1}$ be such that $x\neq 0$ and $\underline{\Lambda} x^T=0$.
By writing the linear system $\underline{\Lambda} x^T=0$,  we see that $x_2=0$ would successively imply $x_3=0$, \dots , $x_\ell=0$. Hence we must  have $x_2 \neq 0$ and we get a new basis of $\mathbb{R}^\ell$ by  substituting $x\in \{0\} \oplus \mathbb{R}^{\ell-1} \subset\mathbb{R}^\ell$ 
to the second vector of the canonical basis: 
we shall express  in this new basis the symmetric bilinear form in $\mathbb{R}^\ell$ that is given by $\Lambda$ in the canonical basis.
Specifically, we have the following  congruence  where  $\underline{x}:= (x_3,\dots,x_{\ell})$,
$w :=(1,0,\dots,0)\in \mathbb{R}^{\ell-2}$ and  $\underline{\underline{\Lambda}}:= \Lambda(c_3,\dots,c_\ell)$:
\begin{equation} \label{eq:congruence1}
\begin{pmatrix}
 0 & 0 \\
 0 & \underline{\underline{\Lambda}}
\end{pmatrix} 
= \underbrace{\begin{pmatrix}   x_2 & \underline{x} \\  0& I  \end{pmatrix}}_{X^T}
\underbrace{\begin{pmatrix}   c_2  &  w \\  w^T & \underline{\underline{\Lambda}} \end{pmatrix}}_{\underline{\Lambda}}
\underbrace{\begin{pmatrix}  x_2 & 0 \\  \underline{x}^T & I  \end{pmatrix}}_{X}
\end{equation}
Therefore, we also have the following congruence:
\begin{equation} \label{eq:congruence2}
\begin{pmatrix}
c_1 & x_2 & 0 \\
x_2 & 0 & 0 \\
0 & 0 & \underline{\underline{\Lambda}}
\end{pmatrix} 
= \begin{pmatrix} 1 & 0 \\ 0  & X^T  \end{pmatrix} 
\underbrace{\begin{pmatrix} c_1  & v \\ v^T & \underline{\Lambda} \end{pmatrix}}_{\Lambda} 
\begin{pmatrix} 1 & 0 \\ 0 & X  \end{pmatrix}
\end{equation}
We deduce from \eqref{eq:congruence1} that $\zeta (\underline{\underline{\Lambda}}) = \zeta({\underline{\Lambda}})$
and we deduce from \eqref{eq:congruence2} that  $\zeta (\underline{\underline{\Lambda}}) = \zeta({{\Lambda}})$.
\end{proof}

We now recall how signatures of tridiagonal matrices with \emph{integral} coefficients 
can be computed  from Dedekind sums. 
For any pair of coprime integers $p,q$ with $q\neq 0$, define the \emph{Dedekind sum} $\mathfrak{s}(p,q)$ by
$$
\mathfrak{s}(p,q) := \sum_{k=1}^{\vert q\vert -1} \left(\!\!\left(\,\frac{k}{q}\, \right)\!\!\right)  
\cdot \left(\!\!\left(\,\frac{kp}{q}\, \right)\!\!\right) 
$$
where $(\!(-)\!)$ denotes the \emph{sawtooth function}  defined by
$(\!(x)\!) := x- \lfloor x \rfloor -1/2$  for {$x\in \mathbb{R}\setminus\Z$} and $(\!(x)\!):=0$ for $x\in \Z$.
Then the \emph{Dedekind symbol} 
$$
S(p/q):= 12 \operatorname{sgn}(q)\,  \mathfrak{s}(p,q)
$$
is easily seen to satisfy $S(-p/q) = -S(p/q)$ and $S(p/q+1)= S(p/q)$.
Furthermore, we have 
\begin{equation} \label{eq:reciprocity}
S(p/q) + S(q/p) = p/q +q/p +1/pq -3 \operatorname{sgn} (pq);
\end{equation}
see \cite{KM} and references therein.

The following theorem is proved in \cite[(1.12) \& (2.2)]{KM} --- see also \cite{BG}.
For the sake of completeness, we provide  a proof   based on  Proposition \ref{prop:signatures}
and assuming the reciprocity law~\eqref{eq:reciprocity}.

\begin{theorem}[Kirby \& Melvin] \label{th:KM}
Let $c_1,\dots,c_\ell\in \Z$ and write
$$
\begin{pmatrix} \alpha & \beta \\ \gamma & \delta \end{pmatrix} := \operatorname{A}(c_1,\dots,c_\ell).
$$
The signature $\zeta(\Lambda)$ and the trace $\operatorname{tr}(\Lambda)$ 
of $\Lambda:= \Lambda(c_1,\dots,c_\ell)$ are related as follows:
\begin{equation} \label{eq:dd}
3 \zeta(\Lambda) - \operatorname{tr}(\Lambda)  
= \left\{\begin{array}{ll}  S(\alpha/\gamma) - (\alpha+\delta)/\gamma & \hbox{if } \gamma \neq 0, \\
-\beta/\alpha   &  \hbox{if } \gamma =  0.
\end{array}\right.
\end{equation}
\end{theorem}

\begin{proof}
The proof is by induction on $\ell \geq 1$. For $\ell=1$, \eqref{eq:dd}  asserts that
$$
3 \operatorname{sgn}(c_1)- c_1= \left\{\begin{array}{ll}  -S(1/c_1) +2/c_1 & \hbox{if } c_1 \neq 0, \\
0  &  \hbox{if } c_1 =  0, \end{array}\right.
$$
which is trivial for $c_1=0$ and follows from \eqref{eq:reciprocity} for $c_1\neq 0$.
Hence we  suppose that \eqref{eq:dd} holds true for $\underline{\Lambda} := \Lambda(c_2, \dots,c_\ell)$; the $2 \times 2$ matrix associated to $\underline{\Lambda}$ has been computed at \eqref{eq:A}.

Assume that $\gamma=0$. It follows from \eqref{eq:left} that
\begin{eqnarray*}
3\zeta(\Lambda) - \operatorname{tr}(\Lambda) &=& 3\zeta(\underline{\Lambda}) - \operatorname{tr}(\underline{\Lambda}) -c_1.
\end{eqnarray*}
Since $\alpha \delta-\beta\gamma=1$, we must have $\alpha\neq 0$.
Then the induction hypothesis gives 
\begin{eqnarray*}
3\zeta(\Lambda) - \operatorname{tr}(\Lambda)
&=& S(-\gamma/\alpha-c_1) - (\gamma+ \alpha c_1 -\beta)/(-\alpha) -c_1 \ = \ -\beta/\alpha.
\end{eqnarray*}
Assume now that $\gamma \neq 0$. If $\alpha=0$, 
we have $\beta=-\gamma \in \{-1,+1\}$ and it follows from \eqref{eq:left} and the induction hypothesis that
\begin{eqnarray*}
3\zeta(\Lambda) - \operatorname{tr}(\Lambda) &=& 3\zeta(\underline{\Lambda}) - \operatorname{tr}(\underline{\Lambda}) -c_1\\
&=& -(\delta+\beta c_1)/ (\gamma + \alpha c_1 ) -c_1  
\ = \  -\delta/\gamma \ = \ S(\alpha/\gamma) - (\alpha+\delta)/\gamma .
\end{eqnarray*}
If $\alpha \neq 0$,  it follows from \eqref{eq:left} and the induction hypothesis that
\begin{eqnarray*}
3\zeta(\Lambda) - \operatorname{tr}(\Lambda) &=& 3\zeta(\underline{\Lambda}) 
- \operatorname{tr}(\underline{\Lambda}) -3  \operatorname{sgn}(\gamma \alpha) -c_1 \\
&=&  S((\gamma+c_1 \alpha)/(-\alpha)) -(\gamma+c_1 \alpha-\beta)/ (-\alpha)  -3  \operatorname{sgn}(\gamma \alpha) -c_1 \\
&=& -S(\gamma/\alpha) + \gamma/\alpha -\beta/\alpha -3  \operatorname{sgn}(\gamma \alpha) \\
&\stackrel{\eqref{eq:reciprocity}}{=}& S(\alpha/\gamma)- \alpha/\gamma -1/(\gamma\alpha) -\beta/\alpha 
\ = \ S(\alpha/\gamma)- \alpha/\gamma -\delta/\gamma.
\end{eqnarray*}

\vspace{-0.5cm}
\end{proof}

\def\cprime{$'$}
\providecommand{\bysame}{\leavevmode ---\ }
\providecommand{\og}{``}
\providecommand{\fg}{''}
\providecommand{\smfandname}{\&}
\providecommand{\smfedsname}{\'eds.}
\providecommand{\smfedname}{\'ed.}
\providecommand{\smfmastersthesisname}{M\'emoire}
\providecommand{\smfphdthesisname}{Th\`ese}

\end{document}